\documentclass[final,a4paper,11pt]{article}
\usepackage{graphicx}
\usepackage{mathptmx}
\usepackage{amssymb,amsmath, amsfonts,amsthm}
\usepackage{a4wide}
\usepackage{color}

\newcommand{\email}[1]{{\tt #1}}
\newcommand{\R}{\mathbb{R}}

\newcommand{\norm}[1]{\|#1\|}

\newcommand{\dist}[1]{{\rm dist}(#1)}

\newcommand{\mv}{\,\mid\,}
\newcommand{\B}{{\cal B}}

\newcommand{\G}{{\cal G}}
\newcommand{\K}{{\cal K}}

\newcommand{\Sp}{{\cal S}}
\newcommand{\F}{{\cal F}}
\newcommand{\Lag}{{\cal L}}

\newcommand{\longsetto}[1]{\mathop{\longrightarrow}\limits^#1}
\newcommand{\skalp}[1]{\langle #1\rangle}

\newcommand{\xb}{\bar x}
\newcommand{\yb}{\bar y}

\newcommand{\lb}{\bar\lambda}

\newcommand{\AT}[2]{{\textstyle{#1\atop#2}}}

\newcommand{\oo}{o}

\newcommand{\range}{{\rm Range\,}}
\newcommand{\lin}{{\rm lin\,}}
\newcommand{\Span}{{\rm span\,}}

\newcommand{\co}{{\rm conv\,}}
\newcommand{\gph}{\mathrm{gph}\,}
\newcommand{\dom}{\mathrm{dom}\,}
\newcommand{\tto}{\rightrightarrows}
\newcommand{\Limsup}{\mathop{{\rm Lim}\,{\rm sup}}}
\newcommand{\myvec}[1]{\left(\begin{array}{c}#1\end{array}\right)}

\newcommand{\ssstar}{semismooth$^{*}$ }

\newtheorem{theorem}{Theorem}[section]
\newtheorem{proposition}[theorem]{Proposition}
\newtheorem{remark}[theorem]{Remark}
\newtheorem{lemma}[theorem]{Lemma}
\newtheorem{corollary}[theorem]{Corollary}
\newtheorem{definition}[theorem]{Definition}
\newtheorem{example}[theorem]{Example}
\newtheorem{algorithm}{Algorithm}
\newtheorem{assumption}{Assumption}

\title{On a semismooth* Newton method for solving generalized equations}
\author{Helmut Gfrerer\thanks{Institute of Computational Mathematics, Johannes Kepler University
Linz, A-4040 Linz, Austria; \email{helmut.gfrerer@jku.at}}
 \and   Ji\v{r}\'{i} V. Outrata\thanks{Institute of Information Theory and Automation, Academy of
 Sciences of the Czech Republic, 18208 Prague, Czech Republic, and Centre for
              Informatics and Applied Optimization, Federation University of Australia, POB 663,
              Ballarat,  Vic 3350, Australia,  \email{outrata@utia.cas.cz}}}
\date{}

\begin{document}
\maketitle

{\footnotesize
{\bf Abstract.} In the paper, a Newton-type method for the solution of generalized equations (GEs) is
derived, where the linearization concerns both the single-valued and the multi-valued part of the
considered GE. The method is based on the new notion of semismoothness${}^*$ which, together with a
suitable regularity condition, ensure the local superlinear convergence. An implementable version of the
new method is derived for a class of GEs, frequently arising in optimization and equilibrium models.

{\bf Key words.} Newton method, semismoothness${}^*$, superlinear convergence, generalized equation,
coderivatives.

{\bf AMS Subject classification.} 65K10, 65K15, 90C33.
}

\section{Introduction}
Starting in the seventies, we observe a considerable number of works devoted to the solution of
generalized and nonsmooth
 equations via a Newton-type method, cf., e.g., the surveys \cite{IzSo15}  and \cite{KlKum18}, the
 monographs \cite{KlKum02}  and \cite{IzSo14} and the references therein.
 Concerning generalized equations (GEs), first results can be found in the papers of N. Josephy
 \cite{Jo79a},
 \cite{Jo79b}. The idea consists in the linearization of the
 single-valued part of the GE so that in the Newton step one
 solves typically an affine variational inequality or a linear
 complementarity problem. The first Newton method for nonsmooth equations has been suggested
in 1988 in a pioneering paper by B. Kummer \cite{Ku88}. This method, based on generalized derivatives,
has been thereafter worked out to various types of non-smooth equations and can be used, after an
appropriate reformulation, also in the case of some variational inequalities and complementarity
problems, cf.
\cite{FaPa03}.

In 1977, R. Mifflin \cite{Mif77} introduced the notion of {\em semismooth} real-valued function which
plays an important role in nonsmooth optimization, cf. \cite{SchrZo92}. Later, this notion has been
extended to vector-valued mappings (\cite{QiSun93}) and it turned out that this property implies the
first of the two principal conditions required in \cite{Ku88} to achieve superlinear convergence. This
relationship is thoroughly explained in \cite[Chapters 6 and 10]{KlKum02}. As a consequence, one uses the
terminology {\em semismooth Newton method} for a large family of Newton-type methods based on the
conceptual scheme from \cite{Ku88} and tailored to various types of nonsmooth equations.

In connection with the metric subregularity of multifunctions, in \cite{HO01} the semismoothness was
extended to sets
 and in \cite{DoGfrKruOut18} the authors introduced a very similar property for multifunctions via
 a relationship between the
  graph and the directions in the respective directional limiting coderivative. This new property,
  called
   semismoothness$^{*}$ in the present paper, enables us, among other things, to construct a
   semismooth$^{*}$ Newton
    method for GEs, very different from the Josephy-Newton method in \cite{Jo79a}, \cite{Jo79b} and
    all its later extensions and modifications. The principal difference consists in the fact that
    the "linearization"  concerns not only the single-valued part but the whole GE. At the same
    time, this method opens some new possibilities even when applied to nonsmooth equations.

The outline of the paper is as follows. In the preliminary Section 2 one finds the necessary background
from variational analysis together with some useful auxiliary results. In Section 3 we introduce the
semismooth$^{*}$ sets and mappings, characterize them in terms of standard (regular and limiting)
coderivatives and investigate thoroughly their relationship to semismooth sets from
\cite{HO01} and the semismooth vector-valued mappings introduced in \cite{QiSun93}. Moreover, in
this section also some basic classes of \ssstar sets and mappings are  presented. The main results are
collected in Sections 4 and 5. In particular, Section 4 contains the basic conceptual version of the new
method suggested for the numerical solution of the general inclusion 
\[
0 \in F(x),
\]
 where $F:\mathbb{R}^{n} \rightrightarrows \mathbb{R}^{n}$. In this version the ``linearization''
in the Newton step is performed on the basis of the limiting coderivative of $F$.
 In many situations of practical importance, however, $F$ is not \ssstar at the solution.
 Nevertheless, on the basis of a modified regular coderivative it is often possible to construct a
 modification of the limiting coderivative, with respect to which $F$ is \ssstar in a generalized
 sense. This enables us to suggest a  generalized version of the new method which exhibits
 essentially the same convergence properties as the basic one.

Both basic as well as generalized version include the so-called approximation step  in which one computes
an approximative projection of the outcome from the Newton step onto the graph of $F$. This is a big
difference with respect to the Josephy Newton methods.

In Section 5 we apply the generalized variant to the frequently arising GE, where $F$ amounts to the sum
of a smooth mapping and the
 normal-cone mapping related to a non-degenerate constraint system. A suitable
 modification of the regular coderivative is found and it
 is shown that F is \ssstar with respect to the respective modification of the limiting
 coderivative. Finally we derive implementable procedures both
 for the approximation as well as for the Newton step. As a result
 one thus obtains a locally superlinearly convergent Newton-type
 method for a class of GEs without assuming the metric regularity of F.
 As shown by a simple example, the method of Josephy may not be
 always applicable to this class of problems because the
 linearized problems need not have a solution.

Our notation is standard. Given a linear space ${\cal L}$, ${\cal L}^\perp$ denotes its orthogonal
complement and for a closed cone $K$ with vertex at the origin, $K^\circ$ signifies its (negative) polar.
$\Sp_{\R^n}$  stands for the unit sphere in $\R^n$ and $\B_\delta(x)$ denotes the closed ball around $x$
with radius $\delta$. Further, given a multifunction $F$, $\gph F:=\{(x,y)\mv y\in F(x)\}$ stands for its
graph. For an element $u\in\R^n$, $\norm{u}$ denotes its Euclidean norm and $[u]$ is the linear space
generated by $u$. In a product space we use the norm $\norm{(u,v)}:=\sqrt{\norm{u}^2+\norm{v}^2}$. Given
a matrix $A$, we employ the operator norm $\norm{A}$ with respect to the Euclidean norm and the Frobenius
norm $\norm{A}_F$. $Id_s$ is the identity matrix in $\mathbb{R}^s$. Sometimes we write only $Id$.

\section{Preliminaries}
Throughout the whole paper, we will make an extensive use of the following basic notions of modern
variational analysis.
 \begin{definition}\label{DefVarGeom}
 Let $A$  be a closed set in $\mathbb{R}^{n}$ and $\bar{x} \in A$. Then
\begin{enumerate}
 \item [(i)]
 $T_{A}(\bar{x}):=\Limsup\limits_{t\searrow 0} \frac{A-\bar{x}}{t}$
 is the {\em tangent (contingent, Bouligand) cone} to $A$ at $\bar{x}$ and
 $ \widehat{N}_{A}(\bar{x}):=(T_{A}(\bar{x}))^{\circ} $
 is the {\em regular (Fr\'{e}chet) normal cone} to $A$ at $\bar{x}$.
 \item [(ii)]
 $ N_{A}(\bar{x}):=\Limsup\limits_{\stackrel{A}{x \rightarrow \bar{x}}} \widehat{N}_{A}(x)$
 is the {\em limiting (Mordukhovich) normal cone} to $A$ at $\bar{x}$ and, given a direction $d
 \in\mathbb{R}^{n}$,
$ N_{A}(\bar{x};d):= \Limsup\limits_{\stackrel{t\searrow 0}{d^{\prime}\rightarrow
 d}}\widehat{N}_{A}(\bar{x}+ td^{\prime})$
 is the {\em directional limiting normal cone} to $A$ at $\bar{x}$ {\em in direction} $d$.
 \end{enumerate}
\end{definition}

If $A$ is convex, then $\widehat{N}_{A}(\bar{x})= N_{A}(\bar{x})$ amounts to the classical normal cone in
the sense of convex analysis and we will  write $N_{A}(\bar{x})$. By the definition, the limiting normal
cone coincides with the directional limiting normal cone in direction $0$, i.e.,
$N_A(\bar{x})=N_A(\bar{x};0)$, and $N_A(\bar{x};d)=\emptyset$ whenever $d\not\in T_A(\bar{x})$.

In the sequel, we will also employ  the so-called critical cone. In the setting of Definition
\ref{DefVarGeom} with an  given normal $d^{*}\in \widehat{N}_{A}(\bar{x})$, the cone
\[
\K_{A}(\bar{x},d^{*}):= T_{A}(\bar{x}) \cap [d^{*}]^{\perp}
\]
is called the {\em critical cone} to $A$ at $\bar{x}$ with respect to $d^{*}$.

The above listed cones enable us to describe the local behavior of set-valued maps via various
generalized derivatives. Consider a closed-graph multifunction $F:\R^n\tto\R^m$ and the point
$(\xb,\yb)\in \gph F$.

\begin{definition}\label{DefGenDeriv}
\begin{enumerate}
\item [(i)]
The multifunction $D F(\xb,\yb):\mathbb{R}^{n} \rightrightarrows\mathbb{R}^{m}$, defined by
\[
DF(\xb,\yb)(u):= \{v \in \mathbb{R}^{m}| (u,v)\in T_{\gph F}(\xb,\yb)\}, u \in
\mathbb{R}^{n}
\]
 is called the {\em graphical derivative} of $F$ at $(\xb,\yb)$.
\item[(ii)]
 The multifunction $\widehat D^\ast F(\xb,\yb ):
 \mathbb{R}^{m}\rightrightarrows\mathbb{R}^{n}$, defined by
\[
\widehat D^\ast F(\xb,\yb )(v^\ast):=\{u^\ast\in \mathbb{R}^{n} | (u^\ast,- v^\ast)\in \widehat
N_{\gph F}(\xb,\yb )\}, v^\ast\in \mathbb{R}^{m}
\]
is called the {\em regular (Fr\'echet) coderivative} of $F$ at $(\xb,\yb )$.
\item [(iii)]
 The multifunction $D^\ast F(\xb,\yb ): \mathbb{R}^{m}\rightrightarrows\mathbb{R}^{n}$,
 defined by
\[
D^\ast F(\xb,\yb )(v^\ast):=\{u^\ast\in \mathbb{R}^{n} | (u^\ast,- v^\ast)\in N_{\gph F}(\xb,\yb )\},
v^\ast\in \mathbb{R}^{m}
\]
is called the {\em limiting (Mordukhovich) coderivative} of $F$ at $(\xb,\yb )$.
\item [(iv)]
 Given a pair of directions $(u,v) \in \mathbb{R}^{n} \times \mathbb{R}^{m}$, the
 multifunction\\
 $D^\ast F((\xb,\yb ); (u,v)):
 \mathbb{R}^{n}\rightrightarrows\mathbb{R}^{m}$, defined by
\begin{equation*}
D^\ast  F((\xb,\yb ); (u,v))(v^\ast):=\{u^\ast \in \mathbb{R}^{n} | (u^\ast,-v^\ast)\in N_{\gph
F}((\xb,\yb ); (u,v)) \}, v^\ast\in \mathbb{R}^{m}
\end{equation*}
is called the {\em directional limiting coderivative} of $F$ at $(\xb,\yb )$ in direction $(u,v)$.
\end{enumerate}
\end{definition}

For the properties of the cones $T_A(\bar{x})$, $\widehat N_A(\bar{x})$ and $N_A(\bar{x})$ from
Definition
\ref{DefVarGeom} and generalized derivatives (i), (ii) and (iii) from Definition \ref{DefGenDeriv}
we refer the interested reader to the monographs \cite{RoWe98} and \cite{Mo18}. The directional limiting
normal cone and coderivative were introduced by the first author in \cite{Gfr13a} and various properties
of these objects can be found also in \cite{GO3} and the references therein. Note that $D^\ast  F(\xb,\yb
)=D^\ast  F((\xb,\yb ); (0,0))$ and that $\dom D^\ast  F((\xb,\yb ); (u,v))=\emptyset$ whenever $v\not\in
DF(\xb,\yb)(u)$.

If $F$ is single-valued, $\yb=F(\xb)$ and we write simply $DF(\xb)$, $\widehat D^*F(\xb)$ and
$D^*F(\xb)$. If $F$ is Fr\'echet differentiable at $\xb$, then
\begin{equation}\label{EqRegCoDerivFDiff}
  \widehat D^*F(\xb)(v^*)=\{\nabla F(\xb)^Tv^*\}
\end{equation}
and, if $F$ is even strictly differentiable at $\xb$, then $D^*F(\xb)(v^*)=\{\nabla F(\xb)^Tv^*\}$.

If the single-valued mapping $F$ is Lipschitzian near $\xb$, denote by $\Omega_F$ the set
\[\Omega_F:=\{x\in\R^n \mv F \mbox{ is differentiable at }x\}.\]
The set
\[\overline\nabla F(\xb):=\{A\in\R^{m\times n}\mv \exists (u_k)\longsetto{{\Omega_F}}\xb\mbox{ such
that }\nabla F(u_k)\to A\}\] is called the {\em B-subdifferential} of $F$ at $\xb$. The {\em Clarke
generalized Jacobian} of $F$ at $\xb$ amounts then to $\co \overline\nabla F(\xb)$. One can prove, see
e.g. \cite[Theorem 9.62]{RoWe98} that
\begin{equation}
  \label{EqConvCoDeriv}\co D^*F(\xb)(v^*)=\{A^Tv^*\mv A\in\co\overline\nabla F(\xb)\}.
\end{equation}
By the definition of $\overline\nabla F(\xb)$ and \eqref{EqRegCoDerivFDiff} we readily obtain
\[\{A^Tv^*\mv A\in \overline\nabla F(\xb)\}\subseteq D^*F(\xb)(v^*).\]

The following iteration scheme, which goes back to Kummer \cite{Ku88}, is an attempt for solving the
nonlinear system $F(x)=0$, where $F:\R^n\to\R^n$ is assumed to be locally Lipschitzian.
\begin{algorithm}[Newton-type method for nonsmooth systems]\label{AlgSemiSmoothNewton}\mbox{ }\\
 1. Choose a starting point $x^{(0)}$, set the iteration counter $k:=0$.\\
  2. Choose $A^{(k)}\in \co\overline\nabla F(x^{(k)})$ and  compute the new iterate
  $x^{(k+1)}=x^{(k)}-{A^{(k)}}^{-1}F(x^{(k)}).$\\
  3. Set $k:=k+1$ and go to 2.
\end{algorithm}
In order to ensure locally superlinear convergence of this algorithm to a zero $\xb$ one has to impose
some assumptions. Firstly, all the matrices ${A^{(k)}}^{-1}$ should be  uniformly bounded, which can be
ensured by the assumption that all matrices $A\in \co \overline\nabla F(\xb)$ are nonsingular. Secondly,
we need an estimate of the form
\[0=F(\xb)=F(x^{(k)})+A^{(k)}(\xb-x^{(k)})+\oo(\norm{\xb-x^{(k)}}).\]
A popular tool how the validity of this estimate could be ensured is the notion of semismoothness
(\cite{Mif77},\cite{QiSun93}). %
\begin{definition}\label{DefSS}
  Let $U\subseteq\R^n$ be nonempty and open. The function
$F:U\to\R^m$ is {\em semismooth} at $\xb\in U$, if it is  Lipschitz near $\xb$ and if
\[\lim_{\AT{A\in\co\overline\nabla F(\xb+tu')}{u'\to u,\ t\downarrow 0}}Au'\]
exists for all $u\in\R^n$. If F is semismooth at all $\xb\in U$, we call $F$ {\em semismooth on} $U$.
\end{definition}

Given a closed convex cone $K\subset\R^n$ with vertex at the origin, then
\[\lin K:=K\cap (-K)\]
denotes the {\em lineality} space of $K$, i.e., the largest linear space contained in $K$. Denoting by
$\Span K$ the linear space spanned by $K$, it holds that
\[\Span K=K+(-K),\ (\lin K)^\perp=\Span K^\circ, \ (\Span K)^\perp = \lin K^\circ.\]

A subset $C'$ of a convex set $C\subset\R^n$ is called a {\em face} of $C$, if it is convex and if for
each line segment $[x,y]\subseteq C$ with $(x,y)\cap C'\not=\emptyset$ one has $x,y\in C'$. The faces of
a polyhedral convex cone $K$ are exactly the sets of the form
\[\F=K\cap[v^*]^\perp\ \mbox{for some}\ v^*\in K^\circ.\]

\begin{lemma}\label{LemLinealityD}
  Let $D\subset \R^s$ be a convex polyhedral set. For every pair $( d,\lambda)\in\gph N_D$ there
  holds
  \begin{align}\label{EqLinealitySp}&\lin T_D( d)=\lin \K_D( d,\lambda)\subseteq \K_D(
  d,\lambda)\subseteq T_D(d),\\
  \label{EqPolarLinealitySp}&N_D( d)\subseteq \K_D( d,\lambda)^\circ \subseteq (\lin T_D( d))^\perp
  = \Span N_D( d).
  \end{align}
  Furthermore, for every $(\bar d,\lb)\in\gph N_D$ there is a neighborhood $U$ of $(\bar d,\lb)$
  such that for every $(d,\lambda)\in\gph N_D\cap U$ there is a face $\F$ of the critical cone
  $\K_D(\xb,\lb)$ such that $\lin T_D(d)=\Span \F$ and consequently $\Span N_D(d)=(\Span
  \F)^\perp$.
\end{lemma}
\begin{proof}
  For every $w\in \lin T_D(d)$ we have $\pm w\in T_D( d)$ and therefore $\pm\skalp{\lambda,w}\leq
  0$ because of $\lambda\in N_D(d)$. This yields $\skalp{\lambda,w}=0$ and consequently
\[\lin T_D(d)\subseteq \K_D(d,\lambda)\subseteq T_D( d)\]
and, by dualizing, \eqref{EqPolarLinealitySp} follows. Since we also have $\K_D(d,\lambda)\subseteq T_D(
d)$, we obtain $\lin\K_D( d,\lambda)\subseteq \lin T_D( d)\subseteq \lin\K_D( d,\lambda)$ implying
\eqref{EqLinealitySp}.

 By \cite[Lemma 4H.2]{DoRo14} there is a neighborhood $U$ of $(\bar d,\lb)$ such that for every
 $(d,\lambda)\in \gph N_D\cap U$ there are two faces $\F_2\subseteq\F_1$ of $\K_D(\bar d,\lb)$ such
 that $\K_D(d, \lambda)=\F_1-\F_2$. We claim that $\lin (\F_1-\F_2)=\F_2-\F_2$. The inclusion $\lin
 (\F_1-\F_2)\supseteq\F_2-\F_2$ trivially holds since $\F_2-\F_2=\Span \F_2$ is a subspace. Now
 consider $w\in  \lin (\F_1-\F_2)=(\F_1-\F_2)\cap (\F_2-\F_1)$. Then there are $u_1,u_2\in \F_1$
 and $v_1,v_2\in \F_2$ such that $w=u_1-v_1=v_2-u_2$ implying $\frac 12 u_1+\frac 12 u_2=\frac
 12(v_1+v_2)$, i.e., the point $\frac 12(v_1+v_2)\in\F_2$ is the midpoint of the line segment
 connecting $u_1,u_2\in\F_1\subseteq \K_D(g(\xb),\lb)$. Since $\F_2$ is  a face of
 $\K_D(g(\xb),\lb)$, $u_1, u_2\in\F_2$ follows and thus $w\in \F_2-\F_2$. Thus our claim holds true
 and from \eqref{EqLinealitySp} we obtain $\lin T_D(d) =\lin\K_D(d,\lambda)=\F_2-\F_2=\Span \F_2$.
 This completes the proof of the lemma.
\end{proof}

\section{\label{SecSS}On \ssstar sets and mappings}
\begin{definition}\label{DefSemiSmooth}
\begin{enumerate}
\item  A set $A\subseteq\R^s$ is called {\em \ssstar} at a point $\xb\in A$ if for all $u\in
    \R^s$ it holds
  \begin{equation}\label{EqSemiSmoothSet}\skalp{x^*,u}=0\ \forall x^*\in N_A(\xb;u).
  \end{equation}
\item
A set-valued mapping $F:\R^n\tto\R^m$ is called {\em \ssstar} at a point $(\xb,\yb)\in\gph F$, if
$\gph F$ is \ssstar at $(\xb,\yb)$, i.e., for all $(u,v)\in\R^n\times\R^m$ we have
\begin{equation}\label{EqSemiSmooth}
\skalp{u^*,u}=\skalp{v^*,v}\ \forall (v^*,u^*)\in\gph D^*F((\xb,\yb);(u,v)).
\end{equation}
\end{enumerate}
\end{definition}

In the above definition the   \ssstar sets and mappings have been defined via directional limiting normal
cones and coderivatives. In some situations, however, it is convenient to make use of equivalent
characterizations in terms of standard (regular and limiting) normal cones and coderivatives,
respectively.

\begin{proposition}\label{PropCharSemiSmoothSets}Let $A\subset \R^s$ and $\xb\in A$ be given. Then
the following three statements are equivalent.
\begin{enumerate}
\item[(i)] $A$ is \ssstar at $\xb$.
\item[(ii)] For every $\epsilon>0$ there is some $\delta>0$ such that
\begin{equation}\label{EqCharSemiSmoothSetReg}
\vert \skalp{x^*,x-\xb}\vert\leq \epsilon \norm{x-\xb}\norm{x^*}\ \forall x\in \B_\delta(\xb)\
\forall x^*\in  \widehat N_A(x);\end{equation}
\item[(iii)] For every $\epsilon>0$ there is some $\delta>0$ such that
  \begin{equation}\label{EqCharSemiSmoothSetLim}
\vert \skalp{x^*,x-\xb}\vert\leq \epsilon \norm{x-\xb}\norm{x^*}\ \forall x\in \B_\delta(\xb)\
\forall x^*\in N_A(x).
\end{equation}
\end{enumerate}
\end{proposition}

\begin{proof}
Assuming  that $A$ is not \ssstar at $\xb$, there is $u\not=0$, $0\not=u^*\in N_A(\xb;u)$   such that
$\epsilon':=\vert \skalp{u^*,u}\vert >0$. By the definition of  directional limiting normals there are
sequences $t_k\downarrow 0$, $u_k\to u$, $u_k^*\to u^*$ such that $u_k^*\in \widehat N_A(\xb+t_ku_k)$.
Then for all $k$ sufficiently large we have $\vert \skalp{u_k^*,u_k}\vert
>\epsilon'/2$ implying
\begin{eqnarray*}\vert \skalp{u_k^*,(\xb+t_ku_k)-\xb}\vert >\frac{\epsilon'}2 t_k =
\frac{\epsilon'}{2\norm{u_k^*}\norm{u_k}}\norm{(\xb+t_ku_k)-\xb}\norm{u_k^*}.\end{eqnarray*}
Hence statement (ii) does not hold for $\epsilon=\epsilon'/\big(4\norm{u^*}\norm{u}\big)$ and the
implication (i)$\Rightarrow$(ii) is shown.

In order to prove the reverse implication we assume that (ii) does not hold, i.e., there is some
$\epsilon>0$ together with sequences $x_k\to\xb$ and $x_k^*$ such that $x_k^*\in\widehat N_A(x_k)$ and
\[\vert \skalp{x_k^*,x_k-\xb}\vert> \epsilon \norm{x_k-\xb}\norm{x_k^*}\]
holds for all $k$. It follows that $x_k-\xb\not=0$ and $x_k^*\not=0$ $\forall k$ and, by possibly passing
to a subsequence, we can assume that the sequences $(x_k-\xb)/\norm{x_k-\xb}$ and $x_k^*/\norm{x_k^*}$
converge to some $u$ and $u^*$, respectively. Then $u^*\in N_A(\xb;u)$ and
\[\vert \skalp{u^*,u}\vert =\lim_{k\to\infty}\frac{\vert
\skalp{x_k^*,x_k-\xb}\vert}{\norm{x_k-\xb}\norm{x_k^*}}>\epsilon\]
showing that $A$ is not \ssstar at $\xb$. This proves the implication (ii)$\Rightarrow$(i).

Finally, the equivalence between (ii) and (iii) is an immediate consequence of the definition of limiting
normals.
\end{proof}

By simply using Definition  \ref{DefSemiSmooth} (part 2) we obtain from Proposition
\ref{PropCharSemiSmoothSets} the following corollary.
\begin{corollary}\label{CorCharSemiSmooth}Let $F:\R^n\tto\R^m$ and $(\xb,\yb)\in \gph F$ be given.
Then the following three statements are equivalent
\begin{enumerate}
\item[(i)] $F$ is \ssstar at $(\xb,\yb)$.
\item[(ii)] For every $\epsilon>0$ there is some $\delta>0$ such that
\begin{equation}\label{EqCharSemiSmoothReg}
\hspace{-1cm}\vert \skalp{x^*,x-\xb}-\skalp{y^*,y-\yb}\vert\leq \epsilon
\norm{(x,y)-(\xb,\yb)}\norm{(x^*,y^*)}\ \forall(x,y)\in \B_\delta(\xb,\yb)\ \forall
(y^*,x^*)\in\gph \widehat D^*F(x,y).\end{equation}
\item[(iii)] For every $\epsilon>0$ there is some $\delta>0$ such that
  \begin{equation}\label{EqCharSemiSmoothLim}
\hspace{-1cm}\vert \skalp{x^*,x-\xb}-\skalp{y^*,y-\yb}\vert\leq \epsilon
\norm{(x,y)-(\xb,\yb)}\norm{(x^*,y^*)}\ \forall(x,y)\in \B_\delta(\xb,\yb)\ \forall
(y^*,x^*)\in\gph D^*F(x,y).
\end{equation}
\end{enumerate}
\end{corollary}

On the basis of Definition \ref{DefSemiSmooth},  Proposition \ref{PropCharSemiSmoothSets} and Corollary
\ref{CorCharSemiSmooth} we may now specify some fundamental classes of \ssstar sets and mappings.

\begin{proposition}\label{PropSSConSet}
Let $A \subset \mathbb{R}^{s}$ be a closed convex set. Then $A$ is \ssstar at each $\bar{x} \in A$.
\end{proposition}
\begin{proof}
Since $N_{A}(\bar{x};u)= \{x^{*}\in N_{A}(\bar{x})|\langle x^{*},u\rangle=0 \}$ by virtue of
\cite[Lemma 2.1]{Gfr14a}, the statement follows immediately from the definition.
\end{proof}

\begin{proposition}
  Assume that we are given closed sets $A_i\subset\R^s$, $i=1,\ldots p$, and $\xb\in
  A:=\bigcup_{i=1}^p A_i$. If the sets $A_i$, $i\in \bar I:=\{j\mv\xb\in A_j\}$, are \ssstar at
  $\xb$, then so is the set $A$.
\end{proposition}
\begin{proof}
  Fix any $\epsilon>0$ and choose according to Proposition \ref{PropCharSemiSmoothSets}
  $\delta_i>0$, $i\in\bar I$, such that for every $i\in\bar I$, every $x\in \B_{\delta_i}(\xb)$ and
  every $x^*\in \widehat N_{A_i}(x)$ there holds
  \begin{equation*}\vert\skalp{x^*,x-\xb}\vert\leq \epsilon \norm{x^*}\norm{x-\xb}.\end{equation*}
  Since the sets $A_i$, $i=1,\ldots,p$, are assumed to be closed, there is some
  $0<\delta\leq\min\{\delta_i\mv i\in\bar I\}$ such that
  \[I(x):=\{j\mv x\in A_j\}\subset \bar I\ \forall x\in \B_\delta(\xb).\]
  Using the identity $\widehat N_A(x)=\bigcap_{i\in I(x)}\widehat N_{A_i}(x)$ valid for every $x\in
  A$ it follows that \eqref{EqCharSemiSmoothSetReg} holds. Thus the assertion follows from
  Proposition \ref{PropCharSemiSmoothSets}.
\end{proof}

Thus, in particular, the union of finitely  many closed convex sets is \ssstar at every point. We obtain
that
\begin{enumerate}
    \item A closed convex multifunction $F:\R^n\tto\R^m$ is \ssstar at every point
        $(\xb,\yb)\in\gph F$.
    \item A polyhedral multifunction $F:\R^n\tto\R^m$ is \ssstar at every point
        $(\xb,\yb)\in\gph F$. In particular, for every convex polyhedral set $D\subset\R^s$ the
        normal cone mapping $N_D$ is \ssstar at every point of its graph.
  \end{enumerate}

  Since the semismoothness$^{*}$ of mappings is defined via the graph, it follows from Corollary
  \ref{CorCharSemiSmooth} that $F:\mathbb{R}^{n} \rightrightarrows \mathbb{R}^{m}$ is \ssstar at
  $(\bar{x},\bar{y})\in
  \gph F $ if and only if $F^{-1}: \mathbb{R}^{m} \rightrightarrows \mathbb{R}^{n}$ is \ssstar at
  $(\bar{y},\bar{x})$. Indeed, the relation (\ref{EqCharSemiSmoothLim}) can be rewritten as

  \[
  \begin{split}
  | \langle y^{*},y-\bar{y}\rangle + \langle x^{*},x-\bar{x}\rangle | \leq \varepsilon \|
  (y,x)-(\bar{y},\bar{x})\|~\| (y^{*},x^{*})\| ~ & \forall (y,x)\in \B_{\delta}(\bar{y},\bar{x})
  \\
  & \forall (-x^{*},-y^{*})\in \gph D^{*}F^{-1}(y,x),
  \end{split}
  \]
  which is, in turn, is equivalent to the semismoothness$^{*}$ of $F^{-1}$ at $(\bar{y},\bar{x})$.

  In some cases of practical importance one  has
  \[
  F(x)=f(x)+Q(x),
  \]
 where $f: \mathbb{R}^{n}  \rightarrow \mathbb{R}^{n}$ is continuously differentiable and $Q:
 \mathbb{R}^{n} \rightrightarrows  \mathbb{R}^{n}$ is a closed-graph multifunction.

 \begin{proposition}
Let $\bar{y}\in F(x)$ and $Q$  be \ssstar at $(\bar{x},\bar{y}- f(\bar{x}))$. Then $F$ is \ssstar at
$(\bar{x},\bar{y})$.
  \end{proposition}

  \begin{proof}
 Let $(u,v)$ be an arbitrary pair of directions and $u^{*}\in D^{*}F((\bar{x},\bar{y});
 (u,v))(v^{*})$ . By virtue of \cite[formula (2.4)]{GO3} it holds that
 \[
 D^{*}F((\bar{x},\bar{y}); (u,v))(v^{*})=\nabla
 f(\bar{x})^{T}v^{*}+D^{*}Q((\bar{x},\bar{y}-f(\bar{x}));(u,v-\nabla f(\bar{x})u))(v^{*}).
 \]
 Thus, $\langle u^{*},u \rangle = \langle \nabla f (\bar{x})^{T}v^{*}+ z^{*},u\rangle $ with some
 $z^{*}\in D^{*}Q((\bar{x},\bar{z}); (u,w))(v^{*})$, where $\bar{z}=\bar{y}- f(x)$ and $w=v-\nabla
 f(\bar{x})u$. It follows that
 \[
 \langle u^{*},u\rangle = \langle v^{*}, \nabla f (\bar{x})u\rangle + \langle z^{*},u \rangle =
 \langle v^{*}, \nabla f(\bar{x})u \rangle + \langle v^{*},w \rangle
 \]
 due to the assumed semismoothness$^{*}$ of $Q$ at $(\bar{x},\bar{y} - f(\bar{x}))$. We conclude
 that $\langle u^{*},u \rangle = \langle v^{*},v\rangle  $ and the proof is complete.
  \end{proof}
 From this statement and the previous development we easily deduce that the solution map
 $S:y
\mapsto x$, related to the canonically perturbed GE
\[ y \in f(x)+N_{\Gamma}(x) \]
 is \ssstar at any $(\bar{y},\bar{x}) \in \gph S$ provided $\Gamma$ is convex polyhedral. Results
of this sort in terms of the standard semismoothness property can be found, e.g., in \cite[Theorems 6.20
and 6.21]{OKZ}.

Let us now figure out the relationship of semismoothness$^{*}$ and the classical semismoothness in case
of single-valued mappings (Definition \ref{DefSS}). To this purpose note that for a continuous
single-valued mapping $F:\R^n\to\R^m$ condition \eqref{EqCharSemiSmoothLim} is equivalent to the
requirement 

\begin{equation}\label{EqCharSemiSmoothSingle}\vert
\skalp{x^*,x-\xb}-\skalp{y^*,F(x)-F(\xb)}\vert\leq \epsilon
\norm{(x,F(x))-(\xb,F(\xb))}\norm{(x^*,y^*)}\ \forall x\in \B_\delta(\xb)\ \forall (y^*,x^*)\in\gph
D^*F(x).\end{equation}

\begin{proposition}\label{PropNewtondiff}Assume that $F:\R^n\to \R^m$ is a single-valued mapping
which is Lipschitzian near $\xb$. Then the following two statements are equivalent.
\begin{enumerate}
  \item[(i)] $F$ is \ssstar at $\xb$.
  \item[(ii)] For every $\epsilon>0$ there is some $\delta>0$ such that
  \begin{equation}\label{EqCharSemiSmoothLipsch}
    \norm{F(x)-F(\xb)-C(x-\xb)}\leq\epsilon\norm{x-\xb}\ \forall x\in \B_\delta(\xb)\ \forall
    C\in\co \overline\nabla F(x).
  \end{equation}
\end{enumerate}
\end{proposition}
\begin{proof}
Let $L$ denote the modulus of Lipschitz continuity of $F$ in some neighborhood of $\xb$. In order to show
the implication (i)$\Rightarrow$(ii), fix any $\epsilon'>0$ and choose $\delta>0$ such that
\eqref{EqCharSemiSmoothSingle} holds with $\epsilon=\epsilon'/(1+L^2)$. Consider
$x\in\B_\delta(\xb)$, $C\in \co \overline\nabla  F(x)$ and choose $y^*\in\Sp_{\R^m}$ with
\[\norm{F(x)-F(\xb)-C(x-\xb)}=\skalp{y^*,F(x)-F(\xb)-C(x-\xb)}.\]
By \eqref{EqConvCoDeriv} there holds $C^Ty^*\in \co D^*F(x)(y^*)$ and therefore, by the Carath\'{e}odory
Theorem, there are elements $x_i^*\in D^*F(x)(y^*)$ and scalars $\alpha_i\geq 0$, $i=1,\ldots,N$, with
$\sum_{i=1}^N\alpha_i=1$ and $C^Ty^*=\sum_{i=1}^N\alpha_i x_i^*$. It follows from
\eqref{EqCharSemiSmoothSingle} that
\begin{eqnarray*}
  \norm{F(x)-F(\xb)-C(x-\xb)}&=&\skalp{y^*,F(x)-F(\xb)-C(x-\xb)}=\skalp{y^*,F(x)-F(\xb)}-\skalp{C^Ty^*,x-\xb}\\
  &=&\sum_{i=1}^N\alpha_i\big(\skalp{y^*,F(x)-F(\xb)}-\skalp{x_i^*,x-\xb}\big)\\
  &\leq&\sum_{i=1}^N\alpha_i\epsilon\norm{(x,F(x))-(\xb,F(\xb))}\norm{(x^*_i,y^*)}\leq\epsilon(1+L^2)\norm{x-\xb}\\
  &=&\epsilon'\norm{x-\xb},
\end{eqnarray*}
where we have taken into account $\norm{x_i^*}\leq L\norm{y^*}=L$ and $\norm{F(x)-F(\xb)}\leq
L\norm{x-\xb}$. This inequality justifies \eqref{EqCharSemiSmoothLipsch} and the implication
(i)$\Rightarrow$(ii) is verified.

Now let us show the reverse implication. Let $\epsilon>0$ and choose $\delta>0$ such that
\eqref{EqCharSemiSmoothLipsch} holds. Consider $x\in\B_\delta(\xb)$ and $(y^*,x^*)\in\gph D^*F(x)$.
Then by \eqref{EqConvCoDeriv} there is some $C\in \co \overline\nabla  F(x)$ such that $x^*\in C^Ty^*$
and we obtain
\begin{eqnarray*}\vert \skalp{x^*,x-\xb}-\skalp{y^*,F(x)-F(\xb)}\vert&=&\skalp{y^*,
C(x-\xb)-(F(x)-F(\xb))}\leq \norm{y^*}\norm{F(x)-F(\xb)-C(x-\xb)}\\ &\leq&
\norm{y^*}\epsilon\norm{x-\xb}\leq \epsilon\norm{(x,F(x))-(\xb,F(\xb))}\norm{(x^*,y^*)}.
\end{eqnarray*}
Thus the implication (ii)$\Rightarrow$(i) is established and the proposition is shown.
\end{proof}

Condition (ii) of Proposition \ref{PropNewtondiff} can be equivalently written in the form that, for any
$C\in \co\overline\nabla  F(\xb
\newline +d)$,
\begin{equation}\label{eq-202}
\norm{F(\xb+d)-F(\xb)-Cd}=\oo(\norm{d})\ \mbox{as}\ d\to 0.
\end{equation}
 In the terminology of \cite[Section 6.4.2]{KlKum02} this condition states that the mapping $x\tto
 \co \overline\nabla  F(x)$ is a {\em Newton map} of $F$ at $\xb$. This is one
 of the conditions used by Kummer \cite{Ku92} for guaranteeing superlinear convergence of a
 generalized Newton method.

If the directional derivative $F'(\xb;\cdot)$ exists (which is the same as the requirement that the
graphical derivative $DF(\xb)(\cdot)$ is single-valued), then we have, cf.  \cite{Sha90}, that
\[F(\xb+d)-F(\xb)-F'(\xb;d)=\oo(\norm{d})\ \mbox{as}\ d\to 0.\]
This relation, together with (\ref{eq-202}) and \cite[Theorem 2.3]{QiSun93} leads now directly to the
following result. 
\begin{corollary}\label{CorSSStar}Assume that $F:\R^n\to \R^m$ is a single-valued mapping which is
Lipschitzian near $\xb$. Then the following two statements are equivalent.
  \begin{enumerate}
    \item [(i)] $F$ is semismooth at $\xb$ (Definition \ref{DefSS}).
    \item[(ii)] $F$ is \ssstar at $\xb$ and $F'(\xb;\cdot)$ exists.
  \end{enumerate}
\end{corollary}
In Definition \ref{DefSemiSmooth} we have started with semismoothness$^{*}$ of sets and extended this
property to mappings via their  graphs. For the reverse direction we may use the distance function. %

\begin{proposition}
  Let $A\subset\R^s$ be closed, $\xb\in A$. Then $A$ is \ssstar at $\xb$ if and only if the
  distance function $d_A$ is \ssstar at $\xb$.
\end{proposition}
\begin{proof}
  The distance function $d_A(\cdot)$ is Lipschitzian with constant $1$ and
  \begin{equation}\label{EqSubdiffDist}\partial d_A(x)=\begin{cases}
    N_A(x)\cap \B_{1}(0)&\mbox{if $x\in A$,}\\
    \frac{x-\Pi_A(x)}{d_A(x)} &\mbox{otherwise,}
   \end{cases}
   \end{equation}
   where $\Pi_A(x):=\{z\in A\mv \norm{z-x}=d_A(x)\}$ denotes the projection on $A$, see, e.g.,
   \cite[Theorem 1.33]{Mo18}. Here, $\partial d_A(x)$ denotes the {\em (basic) subdifferential} of
   the distance function $d_A$ at $x$, see, e.g., \cite[Definition 1.18]{Mo18}.  Further, by
   \cite[Theorem 9.61]{RoWe98} we have
   \[\co\overline\nabla  d_A(x)=\co\partial d_A(x) \]
   for all $x$.

   We first show the implication ''$d_A$ is \ssstar at $\xb\ \Rightarrow\ A$ is \ssstar at $\xb$''.
   For every $x\in A$ and every $0\not=x^*\in N_A(x)$ we have $x^*/\norm{x^*}\in\partial
   d_A(x)\subseteq \co\overline\nabla  d_A(x)$. Thus, if $d_A$ is \ssstar at $\xb$, then it follows
   from Proposition \ref{PropNewtondiff} that for every $\epsilon>0$ there is some $\delta>0$ such
   that for every $x\in\B_\delta(\xb)\cap A$ we have
   \[\vert d_A(x)-d_A(\xb)-\skalp{\frac {x^*}{\norm{x^*}},x-\xb}\vert = \vert \skalp{\frac
   {x^*}{\norm{x^*}},x-\xb}\vert\leq \epsilon \norm{x-\xb}\ \forall 0\not=x^*\in N_A(x).\]
   By taking into account that \eqref{EqCharSemiSmoothSetLim} trivially holds for $x^*=0$ and that
   $N_A(x)=\emptyset$ for $x\in\B_\delta(\xb)\setminus A$, by virtue of Proposition
   \ref{PropCharSemiSmoothSets} the set $A$ is \ssstar at $\xb$.

   In order to show the reverse implication, assume that  $A$ is \ssstar at $\xb$. Fix any
   $\epsilon>0$ and choose $\delta>0$ such that \eqref{EqCharSemiSmoothSetLim} holds. We claim that
   for every $x\in\B_{\delta/2}(\xb)$ and every $x^*\in \co\overline\nabla  d_A(x)$ there holds
   \begin{equation}\label{EqAuxClaim}\vert d_A(x)-d_A(\xb)-\skalp{x^*,x-\xb}\vert \leq 2\epsilon
   \norm{x-\xb}.\end{equation}
   Consider $x\in \B_{\delta/2}(\xb)$. We first show the inequality \eqref{EqAuxClaim} for $x^*\in
   \partial d_A(x)$.
   Indeed, if $x\in A$, then \eqref{EqCharSemiSmoothSetLim} implies
   \[\vert \skalp{x^*,x-\xb}\vert =\vert d_A(x)-d_A(\xb)-\skalp{x^*,x-\xb}\vert\leq
   \epsilon\norm{x^*}\norm{x-\xb}\leq \epsilon\norm{x-\xb}\ \forall x^*\in N_A(x)\cap\B=\partial
   d_A(x).\]
  Otherwise, if $x\not\in A$, for every $x^*\in\partial d_A(x)$ there is some $x'\in\Pi_A(x)$
  satisfying $x^*=(x-x')/d_A(x)$. The vector $x-x'$ is a so-called proximal normal to $A$ at $x'$
  and therefore $x-x'\in \widehat N_A(x')\subset N_A(x')$, see \cite[Example 6.16]{RoWe98}. From
  $\norm{x'-x}\leq \norm{\xb-x}$ we obtain $\norm{x'-\xb}\leq 2\norm{x-\xb}\leq\delta$ and we may
  conclude that
   \begin{eqnarray*}\vert \skalp{x-x',x'-\xb}\vert &=&  \vert \skalp{x-x',x'-x}
   +\skalp{x-x',x-\xb}\vert \\
   &=&\vert -d_A(x)^2 +\skalp{x-x',x-\xb}\vert\leq\epsilon \norm{x-x'}\norm{x'-\xb}\leq  2\epsilon
   d_A(x)\norm{x-\xb}.\end{eqnarray*}
   Dividing by $d_A(x)$ we infer
   \[\vert d_A(x)-\skalp{\frac{x-x'}{d_A(x)}, x-\xb}\vert=\vert
   d_A(x)-d_A(\xb)-\skalp{\frac{x-x'}{d_A(x)}, x-\xb}\vert\leq 2\epsilon\norm{x-\xb}\]
   showing that \eqref{EqAuxClaim} holds true in this case as well.

   Now consider any $x^*\in \co\overline\nabla  d_A(x)=\co\partial d_A(x)$. By the Carath\'{e}odory
   Theorem there are finitely many elements $x_i^*\in\partial d_A(x)$ together with positive
   scalars $\alpha_i$, $i=1,\ldots,N$, such that $\sum_{i=1}^N\alpha_i=1$ and
   $x^*=\sum_{i=1}^N\alpha_i x_i^*$, implying
   \begin{eqnarray*}
     \vert d_A(x)-d_A(\xb)-\skalp{x^*,x-\xb}\vert &=& \vert
     \sum_{i=1}^N\alpha_i\big(d_A(x)-d_A(\xb)-\skalp{x_i^*,x-\xb}\big)\vert\\
     &\leq & \sum_{i=1}^N\alpha_i\vert d_A(x)-d_A(\xb)-\skalp{x_i^*,x-\xb}\vert \leq
     \sum_{i=1}^N\alpha_i 2\epsilon\norm{x-\xb}=2\epsilon\norm{x-\xb}.
   \end{eqnarray*}
   Thus the claimed inequality \eqref{EqAuxClaim} holds for all $x\in\B_{\delta/2}(\xb)$ and all $x^*\in
   \co \overline\nabla  d_A(x)$ and from Proposition \ref{PropNewtondiff} we conclude that $d_A$ is
   \ssstar at $\xb$.
\end{proof}
\begin{remark}Combining Proposition \ref{PropCharSemiSmoothSets} with the formula
\eqref{EqSubdiffDist} implies that a set $A$ is \ssstar at $\xb$ if and only if for every
$\epsilon>0$ there is some $\delta>0$ such that
\[ \vert\skalp{x^*,\frac{x-\xb}{\norm{x-\xb}}}\vert \leq\epsilon\ \forall x\in\B_\delta(\xb)\forall
x^*\in\partial d_A(x).\] From this relation it follows that a set is \ssstar at $\xb$ if and only if it
is semismooth in the sense of \cite[Definition 2.3]{HO01}.
\end{remark}

\section{A \ssstar Newton method}
Given a set-valued mapping $F:\R^n\tto\R^n$ with closed graph, we want to solve the generalized equation
\begin{equation}\label{EqGE}
0\in F(x).
\end{equation}

Given $(x,y)\in\gph F$ we denote by ${\cal A}F(x,y)$ the collection of all pairs of $n\times n$ matrices
$(A,B)$, such that there are $n$ elements $(v_i^*,u_i^*)\in \gph D^*F(x,y)$, $i=1,\ldots, n$, and the
$i$-th row of $A$ and $B$ are ${u_i^*}^T$ and ${v_i^*}^T$, respectively. Further we denote
\[{\cal A}_{\rm reg}F(x,y):=\{(A,B)\in {\cal A}F(x,y)\mv A\mbox{ regular}\}.\]
It turns out that the strong metric regularity of $F$ around $(x,y)$ is a sufficient condition for the
nonemptiness of ${\cal A}_{\rm reg}F(x,y)$. Recall that a set-valued mapping $F:\R^n\tto\R^m$ is {\em
strongly metrically regular} around $( x, y)\in\gph F$ (with modulus $\kappa$), if its inverse $F^{-1}$
has a Lipschitz continuous single-valued localization near $(y,x)$ (with Lipschitz constant $\kappa$),
cf. \cite{DoRo14}.
\begin{theorem}\label{ThExistRegDeriv}
  Assume that $F$ is strongly metrically  regular around $(\hat x,\hat y)\in \gph F$ with modulus
  $\kappa>0$. Then there is an $n\times n$ matrix $C$ with $\norm{C}\leq\kappa$ such that
  $(Id,C)\in {\cal A}_{\rm reg}F(\hat x,\hat y)\not=\emptyset$.
\end{theorem}
\begin{proof}
  Note that $-y^*\in D^*F^{-1}(\hat y,\hat x)(-x^*)$  if and only if $x^*\in D^*F(\hat x,\hat
  y)(y^*)$ , cf. \cite[Equation 8(19)]{RoWe98}. Let $s$ denote the single-valued localization of
  the inverse mapping $F^{-1}$ around $(\hat y,\hat x)$ which is Lipschitzian with modulus $\kappa$
  near $\hat y$. Next take any element $C$ from the B-subdifferential $\overline\nabla s(\yb)$.
  Then $\norm{C}\leq\kappa$ and for any $u^*$ we have $-C^Tu^*\in D^*F^{-1}(\hat y,\hat x)(-u^*)$
  and consequently $u^*\in D^*F(\hat x,\hat y)(C^Tu^*)$. Taking $u^*_i$  as the $i$-th unit vector
  and $v_i^*=C^Tu_i^*$, we obtain that $(Id,C)\in {\cal A}_{\rm reg}F(\hat x,\hat y)$.
\end{proof}
\begin{corollary}
  Let $(\hat x,\hat y)\in \gph F$ and assume that there is $\kappa>0$ and a sequence $(x_k,y_k)$
  converging to $(\hat x,\hat y)$ such that for each $k$ the mapping $F$ is strongly metrically
  regular around $(x_k,y_k)$ with modulus $\kappa$. Then there is an $n\times n$ matrix $C$ with
  $\norm{C}\leq\kappa$ such that $(Id,C)\in {\cal A}_{\rm reg}F(\hat x,\hat y)\not=\emptyset$.
\end{corollary}
\begin{proposition}\label{PropConv} Assume that the mapping $F:\R^n\tto\R^n$ is \ssstar at
$(\xb,0)\in\gph F$. Then for every $\epsilon>0$ there is some $\delta>0$ such that for every $(x,y)\in
\gph F\cap \B_{\delta}(\xb,0)$ and every pair $(A,B)\in {\cal A}_{\rm reg}F(x,y)$ one has
\begin{equation}\label{EqBndNewtonStep}\norm{(x-A^{-1}By)-\xb }\leq
\epsilon\norm{A^{-1}}\norm{(A\,\vdots\, B)}_F\norm{(x,y)-(\xb,0)}.\end{equation}
\end{proposition}
\begin{proof}
  Let $\epsilon>0$ be arbitrarily fixed, choose $\delta>0$ such that \eqref{EqCharSemiSmoothLim}
  holds and consider $(x,y)\in \gph F\cap \B_{\delta}(\xb,0)$ and  $(A,B)\in {\cal A}_{\rm
  reg}F(x,y)$. By the definition of ${\cal A}F(x,y)$ we obtain that the $i$-th component of the
  vector
  $A(x-\xb)-By$ equals to $\skalp{u_i^*,x-\xb}-\skalp{v_i^*,y-0}$ and can be bounded by
  $\epsilon\norm{(x ,y)-(\xb,0)}\norm{(u_i^*,v_i^*)}$ by \eqref{EqCharSemiSmoothLim}. Since the
  Euclidean norm of the vector with components $\norm{(u_i^*,v_i^*)}$ is exactly the Frobenius norm
  of the matrix $(A\,\vdots\, B)$, we obtain
  \[\norm{A(x-\xb)-By}\leq \epsilon \norm{(A\,\vdots\, B)}_F\norm{(x,y)-(\xb,0)}.\]
  By taking into account that
  \[\norm{(x-A^{-1}By)-\xb }=\norm{A^{-1}\big(A(x-\xb)-By\big)}\leq
  \norm{A^{-1}}\norm{A(x-\xb)-By},\]
  the estimate \eqref{EqBndNewtonStep} follows.
\end{proof}
Newton method for solving generalized equations is not uniquely defined in general. Given some iterate
$x^{(k)}$, we cannot expect in general that $F(x^{(k)})\not=\emptyset$ or that $0$ is close to
$F(x^{(k)})$, even if $x^{(k)}$ is close to a solution $\xb$. Thus we perform first some  step which
yields $(\hat x^{(k)},\hat y^{(k)})\in\gph F$ as  an approximate projection of $(x^{(k)},0)$ on $\gph F$.
Further we require that ${\cal A}_{\rm reg}F(\hat x^{(k)},\hat y^{(k)})\not=\emptyset$ and we compute the
new iterate as $x^{(k+1)}=\hat x^{(k)}-A^{-1}B\hat y^{(k)}$ for some $(A,B)\in {\cal A}_{\rm reg}F(\hat
x^{(k)},\hat y^{(k)})$.

\begin{algorithm}[\ssstar Newton-type method for generalized equations]\label{AlgNewton}\mbox{ }\\
 1. Choose a starting point $x^{(0)}$, set the iteration counter $k:=0$.\\
 2. If $0\in F(x^{(k)})$, stop the algorithm.\\
  3. Compute $(\hat x^{(k)},\hat y^{(k)})\in\gph F$ close to $(x^{(k)},0)$ such that ${\cal
  A}_{\rm reg}F(\hat x^{(k)},\hat y^{(k)})\not=\emptyset$.\\
  4. Select $(A,B)\in {\cal A}_{\rm reg}F(\hat x^{(k)},\hat y^{(k)})$ and compute the new iterate
  $x^{(k+1)}=\hat x^{(k)}-A^{-1}B\hat y^{(k)}.$\\
  5. Set $k:=k+1$ and go to 2.
\end{algorithm}

Of course, the heart of this algorithm are steps 3 and 4. We will call step 3 the {\em
approximation step} and step 4 the {\em Newton step}.

Before we continue with the analysis of this algorithm let us consider the Newton step for the special
case of a single-valued smooth mapping $F:\R^n\to\R^n$. We have $\hat y^{(k)}=F(\hat x^{(k)})$ and
$D^*F(\hat x^{(k)})(v^*)=\nabla F(\hat x^{(k)})^Tv^*$ yielding
\[{\cal A}F(\hat x^{(k)},F(\hat x^{(k)}))=\{(B\nabla F(x^{(k)}),B)\mv B\mbox{ is $n\times n$
matrix}\}.\] Thus the requirement $(A,B)\in {\cal A}_{\rm reg}F(\hat x^{(k)},F(\hat x^{(k)}))$ means that
$A=B\nabla F(\hat x^{(k)})$  is regular, i.e., both $B$  and $\nabla F(x^{(k)})$ are regular. Then the
Newton step amounts to
\[x^{(k+1)}=\hat x^{(k)}-(B\nabla F(\hat x^{(k)}))^{-1}BF(\hat x^{(k)}) = \hat x^{(k)}-\nabla
F(\hat x^{(k)})^{-1} F(\hat x^{(k)}).\] We see that it coincides with the classical Newton step for
smooth functions $F$. Note that the requirement that $B$ is regular in order to have $(A,B)\in{\cal
A}_{\rm reg}F(\hat x^{(k)},F(\hat x^{(k)}))$ is possibly not needed for general set-valued mappings $F$,
see \eqref{EqAB} below.

Next let us consider the case of a single-valued Lipschitzian mapping $F:\R^n\to\R^n$. As before we have
$\hat y^{(k)}=F(\hat x^{(k)})$ and for every $C\in \overline\nabla F(\hat x^{(k)})$ we have $C^Tv^*\in
D^*F(\hat x^{(k)})(v^*)$. Thus
\begin{equation}\label{EqInclCalA}{\cal A}F(\hat x^{(k)},F(\hat
x^{(k)}))\supseteq\bigcup_{C\in\overline\nabla F(\hat x^{(k)})}\{(BC,B)\mv B\mbox{ is an $n\times n$
matrix}\}.\end{equation} Similar as above we have that $(BC,B)\in {\cal A}_{\rm reg}F(\hat x^{(k)},F(\hat
x^{(k)}))$ if and only if both $B$ and $C$ are regular and in this case the Newton step reads as
$x^{(k+1)}=
\hat x^{(k)}-C^{-1} F(\hat x^{(k)})$. Thus the classical semismooth Newton method of \cite{QiSun93},
restricted to the B-subdifferential $\overline \nabla F(\hat x^{(k)})$ instead of the generalized
Jacobian $\co\overline \nabla F(\hat x^{(k)})$, fits into the framework of Algorithm
\ref{AlgNewton}. However, note that the inclusion \eqref{EqInclCalA} will  be strict whenever
$\overline\nabla F(\hat x^{(k)})$ is not a singleton: For every $u_i^*$, $i=1,\ldots,n$ forming the rows
of the matrix $B$ we can take a different $C_i\in \overline\nabla F(\hat x^{(k)})$, $i=1,\ldots,n$, for
generating the rows $C_i^Tu_i^*$ of the matrix $A$. When using such a construction it is no longer
mandatory to require $B$ regular in order to have $(A,B)\in {\cal A}_{\rm reg}F(\hat x^{(k)},F(\hat
x^{(k)}))$ and thus Algorithm \ref{AlgNewton} offers a variety of other possibilities, how the Newton
step can be performed.

Given two reals $L,\kappa>0$ and a solution $\xb$ of \eqref{EqGE}, we denote
\[\G_{F,\xb}^{L,\kappa}(x):=\{(\hat x,\hat y,A,B)\mv \norm{(\hat x-\xb,\hat y)}\leq L\norm{x-\xb},\
(A,B)\in {\cal A}_{\rm reg}F(\hat x,\hat y), \norm{A^{-1}}\norm{(A\,\vdots\,B)}_F\leq\kappa\}.\]
\begin{theorem}\label{ThConvSemiSmmooth1}
  Assume that $F$ is \ssstar at $(\xb,0)\in\gph F$ and assume that there are  $L,\kappa>0$ such
  that for every $x\not\in F^{-1}(0)$ sufficiently close to $\xb$ we have
  $\G_{F,\xb}^{L,\kappa}(x)\not=\emptyset$. Then there exists some $\delta>0$ such that for every
  starting point $x^{(0)}\in\B_\delta(\xb)$ Algorithm \ref{AlgNewton} either stops after
  finitely many iterations at a solution or produces a sequence $x^{(k)}$ which converges
  superlinearly to $\xb$, provided we choose in every iteration $(\hat x^{(k)},\hat y^{(k)},A,B)\in
  \G_{F,\xb}^{L,\kappa}(x^{(k)})$.
\end{theorem}
\begin{proof}
  By Proposition \ref{PropConv}, we can find some $\bar\delta>0$ such that \eqref{EqBndNewtonStep}
  holds with $\epsilon=\frac 1{2L\kappa}$ for all $(x,y)\in \gph F\cap \B_{\delta}(\xb,0)$ and all
  pairs $(A,B)\in {\cal A}_{\rm reg}F(x,y)$. Set $\delta:=\bar\delta/L$ and consider an iterate
  $x^{(k)}\in\B_\delta(\xb)\not\in F^{-1}(0)$. Then
  \[\norm{(\hat x^{(k)},\hat y^{(k)})-(\xb,0)}\leq L\norm{x^{(k)}-\xb}\leq\bar\delta\]
  and consequently
  \[\norm{x^{(k+1)}-\xb}\leq \frac 1{2L\kappa}\norm{A^{-1}}\norm{(A\,\vdots\,B)}_F
  L\norm{x^{(k)}-\xb}\leq \frac 12 \norm{x^{(k)}-\xb}\]
  by Proposition \ref{PropConv}.
  It follows that for every starting point $x^{(0)}\in\B_\delta(\xb)$  Algorithm \ref{AlgNewton}
  either stops after finitely many iterations with a solution or produces a sequence $x^{(k)}$
  converging to $\xb$. The superlinear convergence of the sequence $x^{(k)}$ is now an easy
  consequence of Proposition \ref{PropConv}.
\end{proof}
\begin{remark}\label{RemApprProjection}
    The bound $\norm{(\hat x-\xb,\hat y)}\leq L\norm{x-\xb}$ is in particular fulfilled if $(\hat
    x,\hat y)\in \gph F$ satisfies
    \[\norm{(\hat x-x,\hat y)}\leq \beta\dist{(x,0),\gph F}\]
    with some constant $\beta>0$, because then we have
    \[\norm{(\hat x-\xb,\hat y)}\leq \norm{(\hat x-x,\hat y)}+\norm{x-\xb}\leq
    \beta\dist{(x,0),\gph F}+\norm{x-\xb}\leq (\beta+1)\norm{x-\xb}.\]
\end{remark}
\begin{remark}Note that in case of a single-valued mapping $F:\R^n\to\R^n$ an approximation step of
the form $(\hat x^{(k)}, \hat y^{(k)})=(x^{(k)}, F(x^{(k)}))$ requires
$\norm{(x^{(k)}-\xb,F(x^{(k)}))}\leq L\norm{x^{(k)}-\xb}$, which is in general only fulfilled if $F$ is
calm at $\xb$, i.e., there is a positive real $L'$ such that $\norm{F(x)-F(\xb)}\leq L'\norm{x-\xb}$ for
all $x$ sufficiently near $\xb$.
\end{remark}
\begin{theorem}\label{ThConvSemiSmooth2}
  Assume that the mapping $F$ is both \ssstar at $(\xb,0)$ and  strongly metrically regular around
  $(\xb,0)$. Then all assumptions of Theorem \ref{ThConvSemiSmmooth1} are fulfilled.
\end{theorem}
\begin{proof}
Let $s$ denote the single-valued Lipschitzian localization of $F^{-1}$ around $(0,\xb)$ and let $\kappa$
denote its Lipschitz constant. We claim that for every $\beta\geq 1$ the set
$\G_{F,\xb}^{1+\beta,\sqrt{n(1+\kappa^2)}}(x)\not=\emptyset$ for every $x$ sufficiently close to $\xb$.
Obviously there is a real $\rho>0$ such that $s$ is a single-valued localization of $F^{-1}$ around
$(\hat y,\hat x)$ for every $(\hat x,\hat y)\in\gph F\cap \B_\rho(\xb,0)$ and, since $s$ is Lipschitzian
with modulus $\kappa$, we obtain that $F$ is strongly metrically regular around $(\hat x,\hat y)$ with
modulus $\kappa$. Consider now $x\in \B_{\rho'}(\xb)$ where $\rho'<\rho/(1+\beta)$ and $(\hat x,\hat
y)\in\gph F$ satisfying $\norm{(\hat x-x,\hat y)}\leq \beta \dist{(x,0),\gph F}\leq \beta\norm{x-\xb}$.
Then $\norm{\hat x-\xb,\hat y-0}\leq
\beta\norm{x-\xb}+\norm{(x-\xb,0)}=(1+\beta)\norm{x-\xb}< \rho$ and by Theorem
\ref{ThExistRegDeriv} there is some matric $C$ with $\norm{C}\leq\kappa$ such that $(Id,C)\in {\cal
A}_{\rm reg}F(\hat x,\hat y)$. Since $\norm{(Id\,\vdots\,C)}_F^2=n+\norm{C}_F^2\leq n(1+\norm{C}^2)$, we
obtain $(\hat x,\hat y,Id,C)\in
\G_{F,\xb}^{1+\beta,\sqrt{n(1+\kappa^2)}}(x)\not=\emptyset$.
\end{proof}
To achieve superlinear convergence of the \ssstar Newton method, the conditions of Theorem
\ref{ThConvSemiSmooth2} need not be fulfilled. We now introduce a  generalization of  the concept
of semismoothness${}^*$  which enables us to deal with mappings $F$ that are not \ssstar  at $(\xb,0)$
with respect to the directional limiting coderivative in the sense of Definition \ref{DefSemiSmooth}. Our
approach is motivated by the characterization of semismoothness${}^*$ in Corollary
\ref{CorCharSemiSmooth}. In order to achieve superlinear convergence of Algorithm \ref{AlgNewton}, from
the above analysis it is clear that, in fact,
condition \eqref{EqCharSemiSmoothReg} need not to hold for all $(x,y)\in\gph F\cap
\B_\delta(\xb,0)$ and all elements $(y^*,x^*)\in \gph \widehat D^*F(x,y)$, but only for those
points and those elements from the graph of the regular coderivative which we actually use in the
algorithm. Further, there is no reason to restrict ourselves to (regular) coderivatives, we possibly can
use other objects which are easier to compute.

In order to formalize these ideas we introduce the mapping $\widehat{\cal D}^*F:\gph F\to (\R^n\tto\R^n)$
having the property that for every pair $(x,y)\in\gph F$ the set $\gph \widehat {\cal D}^*F(x,y)$ is a
cone. Further we define the associated limiting mapping ${\cal D}^*F:\gph F\to (\R^n\tto\R^n)$ via
\[ \gph {\cal D}^*F(x,y)=\limsup_{(x',y')\longsetto{{\gph F}}(x,y)}\gph  \widehat {\cal
D}^*F(x',y').\]

\begin{definition}\label{DefSemiSmoothGen}
  The mapping $F:\R^n\tto\R^n$ is called \ssstar at $(\xb,\yb)\in\gph F$ with respect to ${\cal
  D^*}F$ if for every $\epsilon>0$ there is some $\delta>0$ such that
  \begin{equation}\label{EqDefSemiSmoothReg}
\vert \skalp{x^*,x-\xb}-\skalp{y^*,y-\yb}\vert\leq \epsilon \norm{(x,y)-(\xb,\yb)}\norm{(x^*,y^*)}\
\forall(x,y)\in \B_\delta(\xb,\yb)\ \forall (y^*,x^*)\in\gph \widehat {\cal
D}^*F(x,y).\end{equation}
\end{definition}
Given $(x,y)\in\gph F$ we denote by ${\cal A}^{{\cal D}^*}F(x,y)$ the collection of all pairs of $n\times
n$ matrices $(A,B)$, such that there are $n$ elements $(v_i^*,u_i^*)\in \gph {\cal D^*}F(x,y)$,
$i=1,\ldots, n$, and the $i$-th row of $A$ and $B$ are ${u_i^*}^T$ and ${v_i^*}^T$, respectively. Further
we denote
\[{\cal A}_{\rm reg}^{{\cal D}^*}F(x,y):=\{(A,B)\in {\cal A}^{{\cal D}^*}F(x,y)\mv A\mbox{
regular}\}.\] Now we can generalize the previous results by replacing ${\cal A}_{\rm reg}F$ by
${\cal A}_{\rm reg}^{{\cal D}^*}F$.

\begin{algorithm}[Generalized \ssstar Newton-like method for generalized
equations]\label{AlgGenNewton}\mbox{ }\\
 1. Choose a starting point $x^{(0)}$, set the iteration counter $k:=0$.\\
 2. If $0\in F(x^{(k)})$ stop the algorithm.\\
  3. Compute $(\hat x^{(k)},\hat y^{(k)})\in\gph F$ close to $(x^{(k)},0)$ such that ${\cal
  A}_{\rm reg}^{{\cal D}^*}F(\hat x^{(k)},\hat y^{(k)})\not=\emptyset$.\\
  4. Select $(A,B)\in {\cal A}_{\rm reg}^{{\cal D}^*}F(\hat x^{(k)},\hat y^{(k)})$ and compute the
  new iterate $x^{(k+1)}=\hat x^{(k)}-A^{-1}B\hat y^{(k)}.$\\
  5. Set $k:=k+1$ and go to 2.
\end{algorithm}

Given two reals $L,\kappa>0$ and a solution $\xb$ of \eqref{EqGE}, we denote
\[\G_{F,\xb,{\cal D}^*}^{L,\kappa}(x):=\{(\hat x,\hat y,A,B)\mv \norm{(\hat x-\xb,\hat y)}\leq
L\norm{x-\xb},\ (A,B)\in {\cal A}_{\rm reg}^{{\cal D}^*}F(\hat x,\hat y),
\norm{A^{-1}}\norm{(A\,\vdots\,B)}_F\leq\kappa\}.\]
\begin{theorem}\label{ThConvGenSemiSmmooth1}
  Assume that $F$ is \ssstar at $(\xb,0)\in\gph F$ with respect to ${\cal D}^*F$ and assume that
  there are $L,\kappa>0$ such that for every $x\not\in F^{-1}(0)$ sufficiently close to $\xb$ we
  have $\G_{F,\xb,{\cal D}^*}^{L,\kappa}(x)\not=\emptyset$. Then there exists some $\delta>0$ such
  that for every starting point $x^{(0)}\in\B_\delta(\xb)$ Algorithm \ref{AlgGenNewton} either
  stops after
  finitely many iterations at a solution or produces a sequence $x^{(k)}$ which converges
  superlinearly to $\xb$, provided we choose in every iteration $(\hat x^{(k)},\hat y^{(k)},A,B)\in
  \G_{F,\xb,{\cal D}^*}^{L,\kappa}(x^{(k)})$.
\end{theorem}

The proof can be conducted along the same lines as the proof of Theorem \ref{ThConvSemiSmmooth1}. 

\section{Solving generalized equations}
We will now illustrate this generalized method by means of a frequently arising class of problems. We
want to apply Algorithm \ref{AlgGenNewton} to the GE
\begin{equation}\label{EqGEAppl}
  0\in f(x)+ \nabla g(x)^T N_D\big(g(x)\big),
\end{equation}
where $f:\R^n\to\R^n$ is continuously differentiable, $g:\R^n\to\R^s$ is twice continuously
differentiable and $D\subseteq\R^s$ is a convex polyhedral set. Denoting $\Gamma:=\{x\in\R^n\mv g(x)\in
D\}$, we conclude $\nabla g(x)^T N_D\big(g(x)\big)\subseteq \widehat N_\Gamma(x)\subseteq N_\Gamma(x)$,
cf. \cite[Theorem 6.14]{RoWe98}. If in addition some constraint qualification is fulfilled, we also have
$N_\Gamma(x)=\widehat N_\Gamma(x)= \nabla g(x)^TN_D(g(x))$ and in this case
\eqref{EqGEAppl} is equivalent to the GE
\begin{equation}\label{EqFirstOrderOptim}0\in f(x)+\widehat N_\Gamma(x).\end{equation}

Unfortunately, in many situations we cannot apply Algorithm \ref{AlgNewton} directly to the GE
\eqref{EqGEAppl} since this would require to find some $\hat x\in g^{-1}(D)$ close to a given $x$
such that $\dist{0,f(\hat x)+\nabla g(\hat x)^T N_D\big(g(\hat x)\big)}$ is small. This subproblem seems
to be of the same difficulty as the original problem.

A widespread approach is to introduce multipliers and to consider, e.g., the problem
\begin{eqnarray}
\label{EqGEApplJos}\myvec{0\\0}&\in&\tilde F(x,\lambda):=\myvec{f(x)+\nabla
g(x)^T\lambda\\(g(x),\lambda\big)}-\{0\}\times \gph N_D.
\end{eqnarray}
We suggest here another equivalent reformulation
\begin{equation}
  \label{EqGEApplReform}\myvec{0\\0}\in F(x,d):=\myvec{f(x)+\nabla g(x)^TN_D(d)\\g(x)-d}
\end{equation}
which avoids the introduction of multipliers as problem variables. Obviously, $\xb$ solves
\eqref{EqGEAppl} if and only if $(\xb,g(\xb))$ solves \eqref{EqGEApplReform}.

In what follows we define for every $\lambda\in\R^s$ the {\em Lagrangian} $\Lag_\lambda:\R^n\to\R^n$ by
\[\Lag_\lambda(x):=f(x)+\nabla g(x)^T\lambda.\]

Next let us consider the regular coderivative of $F$ at some point $\hat z:=((\hat x,\hat d), (\hat
p^*,g(\hat x)-\hat d))\in\gph F$ and choose any $\hat\lambda\in N_D(\hat d)$ with $\hat p^*=\Lag_{\hat
\lambda}(\hat x)$. If $(x^*,d^*)\in\widehat D^*F(\hat z)(p,q^*)$, we have
\begin{eqnarray*}
  0&\geq&\limsup_{((x,d),(p^*,g(x)-d))\longsetto{{\gph F}}\hat z}\frac{\skalp{x^*,x-\hat
  x}+\skalp{d^*,d-\hat d}-\skalp{p,p^*-\hat p^*}-\skalp{q^*,g(x)-d-(g(\hat x)-\hat
  d)}}{\norm{(x-\hat x,d-\hat d,p^*-\hat p^*,g(x)-d-(g(\hat x)-\hat d))}}\\
  &\geq&\limsup_{\AT{x\to\hat x}{(d,\lambda)\longsetto{{\gph N_D}}(\hat
  d,\hat\lambda)}}\frac{\skalp{x^*,x-\hat x}+\skalp{d^*,d-\hat
  d}-\skalp{p,\Lag_\lambda(x)-\Lag_{\hat \lambda}(\hat x)}-\skalp{q^*,g(x)-d-(g(\hat x)-\hat
  d)}}{\norm{(x-\hat x,d-\hat d,\Lag_\lambda(x)-\Lag_{\hat \lambda}(\hat x),g(x)-d-(g(\hat x)-\hat
  d))}}\\
  &=&\limsup_{\AT{x\to\hat x}{(d,\lambda)\longsetto{{\gph N_D}}(\hat
  d,\hat\lambda)}}\frac{\skalp{x^*-\nabla \Lag_{\hat \lambda}(\hat x)^Tp-\nabla g(\hat
  x)^Tq^*,x-\hat x}+\skalp{d^*+q^*,d-\hat d} -\skalp{\nabla
  g(x)p,\lambda-\hat\lambda}}{\norm{(x-\hat x,d-\hat d,\Lag_\lambda(x)-\Lag_{\hat \lambda}(\hat
  x),g(x)-d-(g(\hat x)-\hat d))}}.
\end{eqnarray*}
 Fixing $(d,\lambda)=(\hat d,\hat\lambda)$, we obtain
\begin{eqnarray*}
0&\geq&\limsup_{x\to\hat x}\frac{\skalp{x^*-\nabla \Lag_{\hat\lambda}(\hat x)^Tp-\nabla g(\hat
x)^Tq^*,x-\hat x}}{\norm{(x-\hat x,0,\Lag_{\hat\lambda}(x)-\Lag_{\hat\lambda}(\hat x), g(x)-g(\hat
x))}}.
\end{eqnarray*}
By our differentiability assumption, $\Lag_{\hat\lambda}$ and $g$ are Lipschitzian near $\hat x$ and
therefore we have
\[x^*=\nabla \Lag_{\hat\lambda}(\hat x)^Tp+\nabla g(\hat x)^Tq^*.\]
Similarly, when fixing $x=\hat x$, we may conclude
\begin{eqnarray*}
  0&\geq&\limsup_{(d,\lambda)\longsetto{{\gph N_D}}(\hat d,\hat\lambda)}\frac{\skalp{d^*+q^*,d-\hat
  d} -\skalp{\nabla g(\hat x)p,\lambda-\hat\lambda}}{\norm{(0,d-\hat d,\nabla g(\hat
  x)^T(\lambda-\hat\lambda),d-\hat d)}}
\end{eqnarray*}
implying $d^*+q^*\in \widehat D^*N_D(\hat d,\hat\lambda)(\nabla g(\hat x)p)$. Thus we have shown the
inclusion
\begin{eqnarray}\label{EqInclRegCoderAppl}\widehat D^*F(\hat z)(p,q^*)&\subseteq& T(\hat x,\hat
d,\hat\lambda)(p,q^*)\\
\nonumber&:=&\big\{(\nabla \Lag_{\hat\lambda}(\hat x)^Tp+\nabla g(\hat x)^Tq^*,d^*)\mv d^*+q^*\in
\widehat D^*N_D(\hat d,\hat\lambda)(\nabla g(\hat x)p)\big\}.
\end{eqnarray}
It is clear from the existing theory on coderivatives that this inclusion is strict in general. In order
to proceed we introduce the following non-degeneracy condition.
\begin{definition}\label{DefNonDegen}
  We say that $(x,d)\in\R^n\times\R^s$ is   {\em non-degenerate with modulus $\gamma>0$} if
  \begin{equation}\label{EqNonDegen}
    \norm{\nabla g(x)^T\mu}\geq\gamma\norm{\mu}\ \forall \mu\in\Span N_D(d).
  \end{equation}
  We simply say that $(x,d)$ is non-degenerate if \eqref{EqNonDegen} holds with some modulus
  $\gamma>0$.
\end{definition}
\begin{remark}
  The point $(\hat x,\hat d)$ is non-degenerate if and only if $\ker \nabla g(\hat x)^T\cap\Span
  N_D(\hat d)=\{0\}$, which in turn is equivalent to $\nabla g(\hat x)\R^n+\lin T_D(\hat d)=\R^s$.
  Thus, $(\hat x,\hat d)$ is non-degenerate if and only if $\hat x$ is a non-degenerate point in
  the sense of \cite[Assumption (A2)]{MoOutRa16} of the mapping $g(x)-(g(\hat x)-\hat d)$ with
  respect to $D$.
  By  \cite[Equation (4.172)]{BonSh00}, this is also related to the  non-degenerate points in the
  sense of \cite[Definition 4.70]{BonSh00} without describing the $C^1$-reducibility of the set
  $D$.
\end{remark}
\begin{remark}
  It is not difficult to show that \eqref{EqInclRegCoderAppl} holds with equality if $(\hat x,\hat
  d)$ is non-degenerate. However, this property is not important for the subsequent analysis.
\end{remark}
\begin{lemma}\label{LemUniqueLambda}
  Consider $((\hat x,\hat d), (\hat p^*,g(\hat x)-\hat d))\in\gph F$ and assume that $(\hat x,\hat
  d)$ is non-degenerate. Then the system
  \begin{equation}\label{EqSystLambda}
   \hat p^*=f(\hat x)+\nabla g(\hat x)^T\lambda(=\Lag_\lambda(\hat x)),\ \lambda\in N_D(\hat d)
  \end{equation}
  has a unique solution denoted by $\hat\lambda(\hat x,\hat d, \hat p^*)$.
\end{lemma}
\begin{proof}
  By the definition of $F$, system \eqref{EqSystLambda} has at least one solution. Now let us
  assume that there are two distinct solutions $\lambda_1\not=\lambda_2$. Then
  $0\not=\lambda_1-\lambda_2\in\Span N_D(\hat d)$ and
  \[\nabla g(\hat x)^T(\lambda_1-\lambda_2)=f(\hat x)+\nabla g(\hat x)^T\lambda_1-(f(\hat x)+\nabla
  g(\hat x)^T\lambda_2)=\hat p^*-\hat p^*=0\]
  contradicting the non-degeneracy of $(\hat x,\hat d)$. Hence the solution to \eqref{EqSystLambda}
  is unique.
\end{proof}
We are now in the position to define the mapping $\widehat{\cal D}^*F$. Given some real $\hat\gamma>0$ we
define
\begin{align}\label{EqHatCalDF}&\widehat{\cal D}^*F(\hat z)(p,q^*)\\
\nonumber&:= \begin{cases}T(\hat x,\hat d,\hat\lambda(\hat x,\hat d,\hat p^*))(p,q^*)&\mbox{if
$(\hat x,\hat d)$ is non-degenerate with modulus $\hat \gamma$},\\
\{(0,0)\}&\mbox{if $(\hat x,\hat d)$ is not non-degenerate with modulus $\hat \gamma$ and
$(p,q^*)=(0,0)$,}\\
\emptyset&\mbox{otherwise}
\end{cases}\end{align}
for every $\hat z:=(\hat x,\hat d,\hat p^*,g(\hat x)-\hat d)\in \gph F$ with $T$ given by
\eqref{EqInclRegCoderAppl}. We neglect in the notation the dependence on $\hat\gamma$ which will be
specified later. %
\begin{theorem}\label{ThSemiSmoothAppl}
The mapping $F$ is \ssstar with respect to   ${\cal D}^*F$ at every point $(\xb,g(\xb),0,0)$.
\end{theorem}
\begin{proof}
  By contraposition. Assume on the contrary that there is a solution $(\xb,g(\xb))$ to
  \eqref{EqGEApplReform} together with $\epsilon>0$ and sequences
  $((x_k,d_k),(p_k^*,g(x_k)-d_k))\longsetto{{\gph F}}((\xb,g(\xb)),(0,0))$, $(x_k^*,d_k^*, p_k,
  q_k^*)\in \gph \widehat{\cal D}^*F((x_k,d_k),(p_k^*,g(x_k)-d_k))$ such that
  \begin{align}\label{EqAux1}\vert\skalp{x_k^*,x_k-\xb}+\skalp{d_k^*,d_k-g(\xb)}-&\skalp{p_k,
  p_k^*} -\skalp{q_k^*,g(x_k)-d_k}\vert\\
  \nonumber&>\epsilon\norm{(x_k-\xb,d_k-g(\xb),p_k^*, g(x_k)-d_k)}\norm{(x_k^*,d_k^*,p_k,q_k^*)} \
  \forall k.
  \end{align}
  We may conclude that $(x_k^*,d_k^*,p_k,q_k^*)\not=(0,0,0,0)$ and consequently $(x_k,d_k)$ is
  non-degenerate with
  modulus $\hat\gamma$. It follows that the sequence $\lambda_k:=\hat\lambda(x_k,d_k,p_k^*)$
  defined by Lemma \ref{LemUniqueLambda} fulfills
  \[\hat \gamma\norm{\lambda_k}\leq \norm{\nabla g(x_k)^T\lambda_k}=\norm{p_k^*-f(x_k)}\]
  and hence it is bounded. By possibly passing to a subsequence we can assume that $\lambda_k$
  converges to some $\lb$. It is easy to see that $\lb\in N_D(g(\xb))$ and $\Lag_{\lb}(\xb)=0$
  and by the definition of $\widehat{\cal D}^*F$ we obtain from \eqref{EqAux1}
  \begin{align*}
    &\vert\skalp{x_k^*,x_k-\xb}+\skalp{d_k^*,d_k-g(\xb)}-\skalp{p_k, p_k^*}
    -\skalp{q_k^*,g(x_k)-d_k}\vert\\
    &=\vert\skalp{\nabla\Lag_{\lambda_k}(x_k)^Tp_k+\nabla g(x_k)^Tq_k^*,
    x_k-\xb}+\skalp{d_k^*+q_k^*, d_k-g(\xb)}-\skalp{q_k^*,g(x_k)-g(\xb)}\\
    &\qquad-\skalp{p_k,\Lag_{\lambda_k}(x_k)-\Lag_{\lb}(\xb)}\vert\\
    &=\vert\skalp{p_k,\Lag_{\lambda_k}(\xb)-\Lag_{\lambda_k}(x_k)+\nabla\Lag_{\lambda_k}(x_k)(x_k-\xb)+(\nabla
    g(x_k)-\nabla g(\xb))^T(\lambda_k-\lb)-\nabla g(x_k)^T(\lambda_k-\lb)}\\
    &\qquad+\skalp{q_k^*,g(\xb)-g(x_k)+\nabla
    g(x_k)(x_k-\xb)}+\skalp{d_k^*+q_k^*,d_k-g(\xb)}\vert\\
    &>\epsilon\norm{(x_k-\xb,d_k-g(\xb),\Lag_{\lambda_k}(x_k),
    g(x_k)-d_k)}\norm{(x_k^*,d_k^*,p_k,q_k^*)}\\
    &\geq \epsilon\norm{(x_k-\xb,d_k-g(\xb),\Lag_{\lambda_k}(x_k))}\norm{(d_k^*,p_k,q_k^*)}.
  \end{align*}
  For all $k$ sufficiently large we have
  \begin{align*}
    &\vert\skalp{p_k,\Lag_{\lambda_k}(\xb)-\Lag_{\lambda_k}(x_k)+\nabla\Lag_{\lambda_k}(x_k)(x_k-\xb)+(\nabla
    g(x_k)-\nabla g(\xb))^T(\lambda_k-\lb)}\\
    &\qquad+\skalp{q_k^*,g(\xb)-g(x_k)+\nabla g(x_k)(x_k-\xb)}\vert\\
    &\leq \frac{\epsilon}2\norm{x_k-\xb}\norm{(p_k,q_k^*)}\leq \frac
    \epsilon2\norm{(x_k-\xb,d_k-g(\xb),\Lag_{\lambda_k}(x_k))}\norm{(d_k^*,p_k,q_k^*)}
  \end{align*}
  implying
  \begin{equation}\label{EqAux2}
    \vert \skalp{d_k^*+q_k^*,d_k-g(\xb)} -\skalp{\nabla g(x_k)p_k,\lambda_k-\lb}\vert> \frac
    \epsilon2\norm{(x_k-\xb,d_k-g(\xb),\Lag_{\lambda_k}(x_k))}\norm{(d_k^*,p_k,q_k^*)}.
  \end{equation}
  Next observe that
  \[\norm{\Lag_{\lambda_k}(x_k)}=\norm{\Lag_{\lambda_k}(x_k)-\Lag_{\lb}(\xb)}=\norm{\nabla
  g(x_k)^T(\lambda_k-\lb)+\Lag_{\lb}(x_k)-\Lag_{\lb}(\xb)}\]
  and let $L>0$ denote some real such that $\norm{\Lag_{\lb}(x_k)-\Lag_{\lb}(\xb)}\leq
  L\norm{x_k-\xb}$ $\forall k$.
  If $\norm{x_k-\xb}<\norm{\nabla g(x_k)^T(\lambda_k-\lb)}/(L+1)$ then we have
  \[\norm{\Lag_{\lambda_k}(x_k)}\geq \norm{\nabla
  g(x_k)^T(\lambda_k-\lb)}-\norm{\Lag_{\lb}(x_k)-\Lag_{\lb}(\xb)}>
  \norm{\nabla g(x_k)^T(\lambda_k-\lb)}/(L+1)\]
  implying
  \[\norm{(x_k-\xb, \Lag_{\lambda_k}(x_k))}\geq \norm{\nabla g(x_k)^T(\lambda_k-\lb)}/(L+1).\]
  Obviously this inequality holds as well when  $\norm{x_k-\xb}\geq\norm{\nabla
  g(x_k)^T(\lambda_k-\lb)}/(L+1)$. Further, by Lemma \ref{LemLinealityD} for every $k$ sufficiently
  large there is a face $\F_k$ of $\K_D(g(\xb),\lb)$ with $\Span N_D(d_k)=(\Span \F_k)^\perp$ and,
  since a convex polyhedral set has only finitely many faces, by possibly passing to a subsequence,
  we may assume that $\F_k= \F$ $\forall k$. Then $\lb=\lim_{k\to\infty}\lambda_k\in(\Span
  \F)^\perp$  and consequently  $\lambda_k-\lb\in (\Span \F)^\perp=\Span N_D(d_k)$. This yields
  $\norm{\nabla g(x_k)^T(\lambda_k-\lb)}\geq \hat\gamma\norm{\lambda_k-\lb}$ by non-degeneracy of
  $(x_k,d_k)$ and we obtain the inequality
  \[\norm{(x_k-\xb,d_k-g(\xb),\Lag_{\lambda_k}(x_k))}\geq
  \min\{\frac{\hat\gamma}{L+1},1\}\norm{(d_k-g(\xb),\lambda_k-\lb)}.\]
  Now let us choose some upper bound $C\geq1$ for the bounded sequence $\norm{\nabla g(x_k)}$ in
  order to obtain
  \[\norm{(d_k^*,p_k,q_k^*)}\geq\frac{\norm{(d_k^*+q_k^*,p_k)}}{\sqrt 2}\geq
  \frac{\norm{(d_k^*+q_k^*,\nabla g(x_k)p_k)}}{C\sqrt 2}.\]
  Thus we derive from \eqref{EqAux2}
  \[\vert \skalp{d_k^*+q_k^*,d_k-g(\xb)} -\skalp{\nabla
  g(x_k)p_k,\lambda_k-\lb}\vert>\frac{\epsilon}{2\sqrt 2 C}
  \min\{\frac{\hat\gamma}{L+1},1\}\norm{d_k-g(\xb),\lambda_k-\lb}\norm{(d_k^*+q_k^*,\nabla
  g(x_k)p_k)}\]
  showing, together with $d_k^*+q_k^*\in\widehat D^*N_D(d_k,\lambda_k)(\nabla g(x_k)p_k)$, that the
  mapping $N_D$ is not \ssstar at $(g(\xb),\lb)$. This contradicts our result from Section \ref{SecSS}
  and the theorem is proven.
\end{proof}
Note that the mapping $F$ will in general not be \ssstar in the sense of Definition
\ref{DefSemiSmooth} at a solution $(\xb,g(\xb))$ to \eqref{EqGEApplReform}, provided $(\xb,g(\xb))$
is not non-degenerate.

It is quite surprising that no constraint qualification is required in Theorem
\ref{ThSemiSmoothAppl}. In fact, there is a constraint qualification hidden in our assumption
because usually we are interested in solutions of \eqref{EqFirstOrderOptim} and here we assume that even a solution to \eqref{EqGEApplReform} is given. Based on Theorem \ref{ThSemiSmoothAppl}, in a
forthcoming paper we will present a locally superlinearly convergent Newton-type algorithm which does not
require, apart from  the solvability of \eqref{EqGEAppl}, any other constraint qualification. In this paper
we want just to demonstrate the basic principles how the approximation step  and the Newton step can be
performed. Therefore, for the ease of presentation,  in the remainder of this section we will impose
\begin{assumption}\label{AssNonDegen}
  $(\xb,g(\xb))$ is a non-degenerate solution to \eqref{EqGEApplReform} with modulus $\bar\gamma$.
\end{assumption}
In the following lemma we summarize two easy consequences of Assumption \ref{AssNonDegen}. Recall that a
mapping $G:\R^n\tto\R^m$ is {\em metrically regular} around $(\xb,\yb)\in\gph G$ if there are
neighborhoods $U$ of $\xb$ and $V$ of $\yb$ along with a positive real $\kappa$ such that
\[\dist{x,G^{-1}(y)}\leq \kappa\dist{y,G(x)}\ \forall (x,y)\in U\times V.\]
\begin{lemma}\label{LemPropNonDegen}
  Assume that Assumption \ref{AssNonDegen} is fulfilled. Then there is a neighborhood $W$ of
  $(\xb,g(\xb))$  such that  all points $(x,d)\in W$  are
  non-degenerate with modulus $\bar\gamma/2$. Further, the mapping $x\tto g(x)-D$ is metrically
  regular around $(\xb,0)$ and the mapping $u\tto g(\xb)+\nabla g(\xb)u-D$ is metrically regular
  around $(0,0)$.
\end{lemma}
\begin{proof}
  We show the first assertion by contraposition. Assume on the contrary that there are sequences
  $(x_k,d_k)\to (\xb,g(\xb))$ and $(\mu_k)$ such that $\mu_k\in \Span N_D(d_k)\cap \Sp_{\R^s}$ and
  $\norm{\nabla g(x_k)^T\mu_k}<\bar\gamma/2$ for all $k$. By possibly passing to a subsequence we
  can assume that $\mu_k$ converges to some $\bar \mu\in \Sp_{\R^s}$ satisfying $\norm{\nabla
  g(\xb)^T\bar\mu}\leq\bar\gamma/2$. Since $D$ is polyhedral we have $N_D(d_k)\subset N_D(g(\xb))$
  for all $k$ sufficiently large implying $\bar\mu\in\Span N_D(g(\xb))$, which contradicts our
  assumption on the modulus of non-degeneracy at $(\xb,g(\xb))$. In order to show the metric
  regularity property of the two mappings just note that Assumption \ref{AssNonDegen} implies
  \[\nabla g(\xb)^T\mu=0, \mu\in N_D(g(\xb))\subset \Span N_D(g(\xb))\ \Rightarrow \mu=0.\]
  Now the assertion follows from \cite[Example 9.44]{RoWe98}.
\end{proof}

We now want to specialize the approximation step and the Newton step for the GE
\eqref{EqGEApplReform}. In this case the approximation step can be performed as follows.
\begin{algorithm}[Approximation step]
  \label{AlgApprStep}\mbox{ }
  Input: $x\in\R^n$.
  \par\hangindent\parindent\noindent
  1. Compute a solution $\hat u$ of the strictly convex quadratic program
      \begin{eqnarray*}QP(x)\qquad\min_{u\in\R^n}&& \frac12 \norm{u}^2+f(x)^Tu\\
    \mbox{subject to}&&g(x)+\nabla g(x)u \in D
    \end{eqnarray*}
    together with an associated  multiplier $\hat\lambda\in N_D(g(x)+\nabla g(x)\hat u)$
    satisfying
    \begin{equation}\label{EqKKTQPApprStep}\hat u+f(x)+\nabla g(x)^T\hat\lambda=\hat
    u+\Lag_{\hat\lambda}(x)=0.\end{equation}
  \par\noindent
  2. Set $\hat x:=x$, $\hat d:=g(x)+\nabla g(x)\hat u$, $\hat p^*:=\Lag_{\hat\lambda}(\hat x)$,
  $\hat y:=(\hat p^*, g(\hat x)-\hat d)$.
\end{algorithm}
Obviously we have $\big((\hat x,\hat d),\hat y\big)\in \gph F$. In the following proposition we state
some properties of the output of Algorithm \ref{AlgApprStep} when the input $x$ is sufficiently close to
$\xb$. We denote by $\bar\lambda:=\hat\lambda(\xb,g(\xb),0)$ the unique multiplier  associated with the
non-degenerate solution $(\xb, g(\xb))$ of \eqref{EqGEApplReform}, cf.  Lemma \ref{LemUniqueLambda}.
\begin{proposition}\label{PropApprStep}
  Under Assumption \ref{AssNonDegen}  there is a positive radius $\rho$ and  positive reals
  $\beta,\beta_u$ and $\beta_\lambda$ such that  for all $x\in \B_\rho(\xb)$ the problem $QP(x)$
  has a unique solution and the output $\hat x,\hat d, \hat\lambda, \hat y$ and $\hat u$ of
  Algorithm \ref{AlgApprStep} fulfills
  \begin{gather}\label{EqUBound}\norm{\hat u}\leq \beta_u\norm{x-\xb}\\
  \label{EqApprStepBound}\norm{\big((\hat x,\hat d),\hat y\big)-\big((\xb,g(\xb)),(0,0)\big)}\leq
  \beta\norm{\hat u}\\
  \label{EqMultBound}\norm{\hat\lambda-\lb}\leq\beta_\lambda\norm{x-\xb}.\end{gather}
  Further, $(\hat x,\hat d)$ is non-degenerate with modulus $\bar\gamma/2$ and $N_D(\hat
  d)\subseteq N_D(g(\xb))$.
  \end{proposition}
\begin{proof}
  Let $\tilde\Gamma(x):=\{u\mv \tilde g(x,u):=g(x)+\nabla g(x)u \in D\}$ denote the feasible region
  of the problem $QP(x)$. By Lemma \ref{LemPropNonDegen} the mapping $u\tto \tilde g(\xb,u)-D$ is
  metrically regular around $(0,0)$. Considering $x$ as a parameter and $u$ as the decision
  variable, by \cite[Corollary 3.7]{GfrMo17a}  the system $\tilde g(x,u)\in D$ has the so-called
  Robinson stability property at $(\xb,0)$ implying $\tilde \Gamma(x)\not=\emptyset$ for all $x$
  belonging to some neighborhood $U'$ of $\xb$. Thus the feasible region of the quadratic program
  $QP(x)$ is not empty and since the objective is strictly convex, for every $x\in U'$ the
  existence of a unique solution $\hat u$ follows. Obviously, $\hat u=0$ is the unique solution of
  $QP(\xb)$. Convexity of the quadratic program $QP(x)$ ensures that $\hat u$ is a solution if and
  only if the first-order optimality condition
  \begin{equation}\label{EqAuxGELinProbl}0\in \tilde f(x,\hat u)+\nabla_u \tilde g(x,\hat
  u)^TN_D(\tilde g(x,\hat u))\end{equation}
  with $\tilde f(x,u):=u+f(x)$ is fulfilled. Defining for every $\lambda\in\R^s$ the linear mapping
  $\tilde \F_\lambda:\R^n\to\R^n$ by $\tilde \F_\lambda v:=\nabla_u\tilde
  f(\xb,0)v+\nabla_u^2\skalp{\lambda^T\tilde g(\cdot)}(\xb,0)v=v$, we obviously have
  $\skalp{\tilde\F_\lambda v,v}=\norm{v}^2>0$ $\forall v\not=0$ and therefore all assumption of
  \cite[Theorem 6.2]{GfrMo19} for the isolated calmness property of the solution map to the
  parameterized variational system \eqref{EqAuxGELinProbl} are fulfilled. Thus there is a positive
  radius $\rho'$  and some constant $\beta_u>0$   such that $\B_{\rho'}\subset U'$  and for every
  $x\in B_{\rho'}(\xb)$ the solution $\hat u$ to $QP(x)$ fulfills the inequality $\norm{\hat
  u}\leq\beta_u\norm{x-\xb}$.
  Setting $L:=\sup\{\norm{\nabla g(x)}\mv x\in\B_{\rho'}(\xb)\}$, we obtain
  \[\norm{\hat d-g(\xb)}\leq \norm{g(x)-g(\xb)}+\norm{\nabla g(x)\hat u}\leq
  L(\norm{x-\xb}+\norm{\hat u})\leq L(1+\beta_u)\norm{x-\xb}\]
  and
  \[\norm{\hat y}=\norm{(\hat p^*, g(\hat x)-\hat d)}=\norm{-(\hat u,\nabla g(x)\hat u)}\leq
  \sqrt{1+L^2}\norm{\hat u}\leq \sqrt{1+L^2}\beta_u\norm{x-\xb}\]
  implying that \eqref{EqApprStepBound} holds with $\beta^2=1+L^2(1+\beta_u)^2+(1+L^2)\beta_u^2.$
  Next we choose $0<\rho\leq \rho'$ such that $\B_\rho(\xb)\times\B_{\beta\rho}(g(\xb))$ is
  contained in the neighborhood $W$ given by Lemma \ref{LemPropNonDegen}. Then $(\hat x,\hat d)$ is
  non-degenerate with modulus $\bar\gamma/2$ and we obtain
  \[\frac{\bar\gamma}2\norm{\hat\lambda}\leq \norm{\nabla g(x)\hat\lambda}=\norm{-\hat u-f(x)}\]
  showing that $\hat\lambda$ remains uniformly  bounded for $x\in \B_{\rho}(\xb)$. Further, since
  $D$ is polyhedral, there is some neighborhood $O$ of $g(\xb)$ such that $N_D(d)\subseteq
  N_D(g(\xb))$ $\forall d\in D\cap O$ and we may assume that $\rho$ is chosen small enough so that
  $\B_{\beta\rho}(g(\xb))\subset O$. Then $\hat\lambda-\lb\in \Span N_D(g(\xb))$ and we obtain
  \begin{eqnarray*}
    \bar\gamma\norm{\hat\lambda-\lb}&\leq&\norm{\nabla g(\xb)^T(\hat\lambda-\lb)}\leq \norm{\nabla
    g(x)^T\hat\lambda -\nabla g(\xb)^T\lb}+\norm{\nabla g(x)^{T}\hat{\lambda}-\nabla g(\xb)^T\hat
    \lambda}\\
    &=&\norm{-f( x)-\hat u+f(\xb)}+\norm{\nabla g(x)-\nabla g(\xb)}\norm{\hat \lambda}\leq
    (L_f+\beta_u+L_{\nabla g}\norm{\hat\lambda})\norm{x-\xb},
  \end{eqnarray*}
  where $L_f$ and $L_{\nabla g}$ denote the Lipschitz moduli  of $f$ and $\nabla g$ in
  $\B_\rho(\xb)$, respectively. This implies \eqref{EqMultBound}.
\end{proof}
Having performed the approximation step, we now turn to the Newton step. We start with the following
auxiliary lemma.
\begin{lemma}\label{LemPrepNewon}
  Let $\hat d\in D$ and let $\hat l:=\dim (\lin T_D(\hat d))$. Then for every $s\times (s-\hat l)$
  matrix $\hat W$, whose columns belong to $N_D(\hat d)$ and form a basis for $\Span N_D(\hat d)$,
  and every $\hat x\in\R^n$ there holds
  \begin{equation}
    \label{EqNewtonStepConstr}\{u\mv \nabla g(\hat x)u\in\lin T_D(\hat d)\}=\{u\mv \hat W^T\nabla
    g(\hat x)u=0\}.
  \end{equation}
  Moreover, if $(\hat x,\hat d)$ is non-degenerate then $\hat W^T\nabla g(\hat x)$ has full row
  rank $s-\hat l$.\\
\end{lemma}
\begin{proof}
\eqref{EqNewtonStepConstr} is an immediate consequence of the relation
\begin{align*}&\nabla g(\hat x)u\in \lin T_D(\hat d)=(\Span N_D(\hat d))^\perp= (\range \hat
W)^\perp\ \Leftrightarrow\
\forall w\in\R^{s-\hat l}\skalp{\hat Ww, \nabla g(\hat x)u}=0\\
& \Leftrightarrow\ \hat W^T\nabla g(\hat x)u=0.
\end{align*}
Now assume that $\hat W^T\nabla g(\hat x)$ does not have full row rank and there is some
$0\not=\mu\in\R^{s-\hat l}$ with $\mu^T\hat W^T\nabla g(\hat x)=0$. Then $0\not=\hat W\mu\in \Span
N_D(\hat d)$ and $\nabla g(\hat x)^T(\hat W\mu)=0$ contradicting the non-degeneracy of $(\hat x,\hat
d)$.
\end{proof}

Assume now that $(\hat x,\hat d)$ is non-degenerate with modulus $\hat\gamma$. One can extract from
\cite[Proof of Theorem 2]{DoRo96} that
\[\widehat N_{\gph N_D}(\hat d,\hat\lambda)=\K_D(\hat d,\hat\lambda)^\circ\times \K_D(\hat
d,\hat\lambda).\] Thus, by \eqref{EqHatCalDF}, for $(p,q^*)\in\R^n\times\R^s$ the set $\hat{\cal
D}F\big((\hat x,\hat d),(\hat p^*, g(\hat x)-\hat d)\big)(p,q^*)$ consists of the elements $(\nabla
\Lag_{\hat\lambda}(\hat x)^Tp+\nabla g(\hat x)^Tq^*,d^*)$ such that
\begin{align}\label{EqCoDerivN_D}&d^*+q^*\in \widehat D^*N_D(\hat d,\hat\lambda)(\nabla g(\hat
x)p)=\begin{cases}\K_D(\hat d,\hat\lambda)^\circ&\mbox{if $-\nabla g(\hat x)p\in \K_D(\hat
d,\hat\lambda)$}\\
\emptyset&\mbox{else.}
\end{cases}
\end{align}
We have to compute suitable matrices $(A,B)\in {\cal A}_{\rm reg}^{{\cal D}^*}F\big((\hat x,\hat d),(\hat
p^*, g(\hat x)-\hat d)\big)$. This is done by choosing suitable elements $(p_i,q_i^*,d_i^*)$,
$i=1,\ldots,n+s$, fulfilling \eqref{EqCoDerivN_D} and setting $A_i,B_i$, the $i$-th row of $A$ and $B$,
respectively, to
\[A_i:=\left(p_i^T\nabla\Lag_{\hat\lambda}(\hat x)+{q_i^*}^T\nabla g(\hat x)\ \vdots\
{d_i^*}^T\right),\quad B_i:= \left(p_i^T\ \vdots\ {q_i^*}^T\right),\ i=1,\ldots,n+s.\]

Denoting $\hat l:=\dim(\lin T_D(\hat d))$, we have $\dim(\Span N_D(\hat d))=s-\hat l$ and we can find an
$s\times (s-\hat l)$ matrix $\hat W$, whose columns belong to $N_D(\hat d)$ and form a basis for $\Span
N_D(\hat d)$, cf. Lemma \ref{LemPrepNewon}. The $(s-\hat l)\times n$ matrix $\hat W^T\nabla g(x)$ has
full row rank $s-\hat l$ and thus we can find  vectors $p_i$, $i=1,\ldots, n-(s-\hat l),$ constituting an
orthonormal basis for $\ker \hat W^T\nabla g(x)$ and set $d_i^*=q_i^*=0$, $i=1,\ldots, n-(s-\hat l)$.
By \eqref{EqLinealitySp} and \eqref{EqNewtonStepConstr} we have $-\nabla g(\hat x)p_i\in \K_D(\hat
d,\hat\lambda)$ and $d_i^*+q_i^*=0\in\K_D(\hat d,\hat\lambda)^\circ$ trivially holds. The next elements
$p_i$, $i=n-(s-\hat l)+1,\ldots, n+s$ are all chosen as $0$. Further we choose the $s-\hat l$ elements
$q_i^*$, $i=n-(s-\hat l)+1,\ldots, n$, as the columns of the matrix $\hat W$ and set $d_i^*=0$. Finally we
set $q_i^*:=-d_i^*:=e_{i-n}$, $i=n+1,\ldots,n+s$, where $e_j$ denotes the $j$-th unit vector.

With this choice, the corresponding matrices $(A,B)\in {\cal A}^{{\cal D}^*}F\big((\hat x,\hat d),(\hat
p^*, g(\hat x)-\hat d)\big)$ are given by
\begin{equation}\label{EqAB}
  A=\left(\begin{array}{ccc}
  \hat Z^T\nabla\Lag_{\hat\lambda}(\hat x)&\vdots&0\\
  \cdots&&\cdots\\
  \hat W^T\nabla g(\hat x)&\vdots&0\\
  \cdots&&\cdots\\
  \nabla g(\hat x)&\vdots&-Id_s
  \end{array}\right),
  \quad B=\left(\begin{array}{ccc}\hat Z^T&\vdots&0\\  \cdots&&\cdots\\
  0&\vdots&\hat W^T\\
    \cdots&&\cdots\\
  0&\vdots&Id_s
  \end{array}\right),
\end{equation}
where $\hat Z$ is the  $n\times (n-(s-\hat l))$ matrix with columns  $p_i$, $i=1,\ldots,n-(s-\hat l)$. In
particular, we have $\hat Z^T\hat Z=Id_{n-(s-\hat l)}$ and $\hat W^T\nabla g(x)\hat Z=0$. Note that the
matrix $B$ in \eqref{EqAB} is certainly not regular.

\begin{lemma}\label{LemInvA}
  Assume that the matrix $G:=\hat Z^T\nabla \Lag_{\hat\lambda}(\hat x)\hat Z$ is regular. Then the
  matrix $A$ in \eqref{EqAB} is regular and
  \begin{equation}\label{EqInvA}A^{-1}=\left(\begin{array}
    {ccccc}\hat ZG^{-1}&\vdots&(Id_n-\hat ZG^{-1}\hat Z^T\nabla\Lag_{\hat \lambda}(\hat
    x))C^\dag&\vdots&0\\
    \cdots&&\cdots&&\cdots\\
    \nabla g(\hat x)\hat ZG^{-1}&\vdots&\nabla g(\hat x)(Id_n-\hat ZG^{-1}\hat Z^T\nabla\Lag_{\hat
    \lambda}(\hat x))C^\dag&\vdots&-Id_s
  \end{array}\right),\end{equation}
  where the $n\times(s-\hat l)$ matrix $C^\dag:=C^T(CC^T)^{-1}$ is the Moore-Penrose inverse of
  $C:=\hat W^T\nabla g(\hat x)$.
\end{lemma}
\begin{proof}
 Follows by the observation that the product of $A$ with the matrix on the right hand side of
 \eqref{EqInvA} is the identity matrix.
\end{proof}
Since $D$ is polyhedral, there are only finitely many possibilities for $N_D(\hat d)$ and we assume that
for identical normal cones we always use the same matrix $\hat W$.

Note that the matrix  $\hat Z$ and consequently also the matrices $G$ and  $G^{-1}$ are not uniquely
given. Let $Z_1,Z_2$ be two $n\times (n-(s-\hat l))$ matrices whose columns form an orthogonal basis of
$\ker C$ and $G_i:=Z_i^T\nabla \Lag_{\hat\lambda}(\hat x)Z_i$, $i=1,2$. Then $Z_2=Z_1V$, where the matrix
$V:=Z_1^TZ_2$ is orthogonal, and consequently
\begin{align*}&G_2=V^TG_1V,\ G_2^{-1}=V^TG_1^{-1}V,\ Z_2G_2^{-1}=Z_1G_1^{-1}V,\
Z_2G_2^{-1}Z_2^T=Z_1G_1^{-1}Z_1^T\\ &\norm{G_2}=\norm{G_1},\ \norm{G_2}_F=\norm{G_1}_F,\
\norm{Z_2G_2^{-1}}=\norm{Z_1G_1^{-1}},\
\norm{Z_2G_2^{-1}}_F=\norm{Z_1G_1^{-1}}_F.
\end{align*}
It follows that the property of invertibility of $G$ (and consequently the invertibility of $A$), the
matrix $\hat ZG^{-1}\hat Z^T$ and the quantity $\norm{A^{-1}}_F\norm{(A\,\vdots\,B)}_F$ are independent
of the particular choice of $\hat Z$. In order to ensure that $A^{-1}$ exists and is bounded, a suitable
second-order condition has to be imposed.
\begin{assumption}\label{AssSecOrd}For every face $\F$ of the critical cone $\K_D(g(\xb),\lb)$
there is a matrix $Z_{\F}$, whose columns form an orthogonal basis of $\{u\mv\nabla g(\xb)u\in \Span
\F\}$, such that the matrix $Z_{\F}^T\Lag_{\lb}(\xb)Z_{\F}$ is regular.
\end{assumption}
In fact, if $Z_{\F}^T\Lag_{\lb}(\xb)Z_{\F}$ is regular, then $Z^T\Lag_{\lb}(\xb)Z$ is regular for every
matrix $Z$ representing the subspace $\{u\mv\nabla g(\xb)u\in \Span \F\}$. %
\begin{remark}In case when $D=\R^s_-$, let $\bar I:=\{i\in\{1,\ldots,s\}\mv g_i(\xb)=0\}$ denote
the index set of active inequality constraints and let $\bar I^+:=\{i\in\bar I\mv \lb_i>0\}$ denote the
index set of positive multipliers. Then the faces of $\K_{\R^s_-}(g(\xb),\lb)$ are exactly the sets
\[\{d\in\R^s\mv d_i=0,\ i\in J, d_i\leq 0, i\in\bar I\setminus J\},\ \bar I^+\subseteq J\subseteq
\bar I.\]
Thus Assumption \ref{AssSecOrd} says that for every index set $\bar I^+\subseteq J\subseteq \bar I$ and
every matrix $Z_J$, whose columns form an orthogonal basis of the subspace $\{u\mv \nabla g_i(\xb)u=0,\
i\in J\}$, the matrix $Z_J^T\Lag_{\lb}(\xb)Z_J$ is regular.
\end{remark}
\begin{proposition}\label{PropCalG}Assume that at the solution $\xb$ of \eqref{EqGEAppl} both
Assumption \ref{AssNonDegen} and Assumption \ref{AssSecOrd} are fulfilled and let $\widehat{\cal D}^*F$
be given by \eqref{EqHatCalDF} with $\hat\gamma=\bar\gamma/2$. Then  there are constants $\tilde
L,\kappa>0$ such that for every $x$ sufficiently close to $\xb$ not solving
\eqref{EqGEAppl} and for every $d\in\R^s_-$ the quadruple $((\hat x,\hat d), (\hat p^*,g(\hat x)-\hat
d), A,B)$ belongs to $\G_{F,(\xb,g(\xb)),{\cal D}^*}^{\tilde L,\kappa}(x, d)$, where $\hat x,\hat d, \hat
p^*$ are the result of Algorithmus \ref{AlgApprStep} and $A, B$ are given by \eqref{EqAB}. In particular,
$\G_{F,(\xb,g(\xb)),{\cal D}^*}^{\tilde L,\kappa}(x, d)\not=\emptyset$.
\end{proposition}
\begin{proof} Let $\rho,\beta,\beta_u,\beta_\lambda$ and $U$ be as in Proposition \ref{PropApprStep}
and Lemma \ref{LemPrepNewon}, respectively, and set $\tilde L:=\beta\beta_u$. By possibly reducing $\rho$
we may assume that $\B_\rho(\xb)\times \B_{\beta_\lambda\rho}(\lb)\subset U$. Then, for every $x\in
\B_\rho(\xb)$ and every $d\in\R^s_-$ we have
\begin{equation}\label{EqAuxTildeL}\norm{(\hat x,\hat d,\hat p^*,g(\hat x)-\hat
d)-(\xb,g(\xb),0,0)}\leq \beta\beta_u\norm{x-\xb}\leq \tilde L\norm{(x,d)-(\xb,g(\xb))}\end{equation} by
Proposition \ref{PropApprStep} and there remains to show that $\norm{A^{-1}}_F\norm{(A\,\vdots\,B)}_F$ is
uniformly bounded for $x$ close to $\xb$. We consider the following possibility for computing a matrix
$\hat Z$, whose columns are an orthonormal basis for a given $m\times n$ matrix $C$. Let $Q$ be an
$n\times n$ orthogonal matrix such that $CQ=(L\,\vdots\,0)$, where $L$ is an $m\times m$ lower triangular
matrix. If ${\rm rank\,} C=m$, then $\hat Z$ can be taken as the last $n-m$ columns of $Q$, cf.
\cite[Section 5.1.3]{GiMur81}. This can be practically done by so-called Householder transformations,
see, e.g., \cite[Section 2.2.5.3]{GiMur81}. When performing the Householder transformations, the signs of
the diagonal elements of $L$ are usually chosen in such a way that cancellation errors are avoided.
However, when modifying the Householder transformations in order to obtain nonnegative diagonal elements
$L_{ii}$, it can be easily seen that the algorithm produces $Q$ and $L$ depending continuously
differentiable on $C$, provided $C$ has full row rank. Since the quantity
$\norm{A^{-1}}_F\norm{(A\,\vdots\,B)}_F$ does not depend on the particular choice of $\hat Z$, we can
assume that $\hat Z$ is computed in such a way.  Now assume that the statement of the proposition does
not hold true. In view of \eqref{EqAuxTildeL} there must be a sequence $x_k$ converging to $\xb$ such
that Algorithm \ref{AlgApprStep} produces with input $x_k$ the quantities $\hat x_k$, $\hat\lambda_k$,
$\hat p_k^*$ and $\hat d_k$  resulting by \eqref{EqAB} in matrices $\hat W_k,\hat Z_k, A_k,B_k$, where either
$A_k$ is singular or $\norm{{A_k}^{-1}}_F\norm{(A_k\,\vdots\,B_k)}_F
\to\infty$ as $k\to\infty$. Since there are only finitely many possibilities for $\hat W_k$ and
there are only finitely many faces of $\K_D(g(\xb),\lb)$, we can assume that $\hat W_k=\hat W$ and $\lin
T_D(\hat d)=\Span \F$ $\forall k$ for some face $\F$ of $\K_D(g(\xb),\lb)$ by Lemma
\ref{LemLinealityD}. In view of \eqref{EqNewtonStepConstr} we have $\{u\mv\nabla g(\xb)u\in \Span
\F\}=\ker (\hat W^T\nabla g(\xb))$ and we can assume that the matrix $Z_\F$ is computed as above
via an orthogonal factorization of the matrix $\hat W^T\nabla g(\xb)$. It follows that $\hat Z_k$
converges to $Z_\F$ and thus $\hat Z_k^T\Lag_{\hat\lambda_k}(\hat x_k)\hat Z_k$ converges to the regular
matrix $Z_{\F}^T\Lag_{\lb}(\xb)Z_{\F}$. Thus for all $k$ sufficiently large the matrices $\hat
Z_k^T\Lag_{\hat\lambda_k}(\hat x_k)\hat Z_k$  are regular and their inverses are uniformly bounded. Since
the matrices $\hat W^T\nabla g(\hat x_k)$ converge to the matrix $\hat W^T\nabla g(\xb)$ having full row
rank, its Moore-Penrose inverses converge to the one of $\hat W^T\nabla g(\xb)$. From Lemma \ref{LemInvA}
we may conclude that the matrices $A_k$ are regular and $\norm{{A_k}^{-1}}_F\norm{(A_k\,\vdots\,B_k)}_F$
remains bounded. Thus the statement of the proposition must hold true.
\end{proof}
We are now in the position to explicitly write down the Newton step. By Algorithm
\ref{AlgGenNewton} the new iterate amounts to $(\hat x,\hat d)+(s_x,s_d)$ with
\[\myvec{s_x\\s_d}=-A^{-1}B\myvec{\hat p^*\\g(\hat x)-\hat d},\]
i.e., $(s_x,s_d)$ solves the linear system
\begin{align*}&\hat Z^T\big(\nabla \Lag_{\hat\lambda}(\hat x)s_x+\Lag_{\hat\lambda}(\hat
x)\big)=0\\ &\hat W^T(g(\hat x)+\nabla g(\hat x)s_x-\hat d)=0 \\ &g(\hat x)+\nabla g(\hat x)s_x-(\hat
d+s_d)=0.
\end{align*}
Note that by the definition of $\hat W$ the second equation can be equivalently written as
\[g(\hat x)+\nabla g(\hat x)s_x-\hat d\in\ker \hat W^T=(\range W)^\perp= (\Span N_D(\hat
d))^\perp=\lin T_D(\hat d).\] It appears that we need not to compute the auxiliary variables $d$, $s_d$
and the columns of $\hat W$ need not necessarily belong to $N_D(\hat d)$.
\begin{algorithm}[Semismooth${}^*$ Newton method for solving
\eqref{EqGEAppl}]\label{AlgNewtonGEAppl}\mbox{ }
\par\hangindent\parindent\hangafter 1\noindent
1. Choose a starting point $x^{(0)}$. Set $k:=0$.
\par \hangindent\parindent\hangafter 1\noindent
2. If $x^{(k)}$ is a solution of \eqref{EqGEAppl}, stop the algorithm.
\par \hangindent2em\hangafter 1\noindent
3. Run Algorithm \ref{AlgApprStep} with input $x^{(k)}$ in order to compute
$\hat\lambda^{(k)}$, $\hat d^{(k)}$ and $\hat p^*{}^{(k)}= \Lag_{\hat \lambda^{(k)}}(x^{(k)})$.
\par \hangindent\parindent\hangafter 1\noindent
4. Set $\hat l^{(k)}=\dim (\lin T_D(\hat d^{(k)}))$ and compute an $s\times (s-\hat l^{(k)})$ matrix
$\hat W^{(k)}$, whose columns form a basis for $\Span N_D(\hat d^{(k)})$ and then an $n\times (n-(s-\hat
l^{(k)}))$ matrix  $\hat Z^{(k)}$, whose columns are an orthogonal basis for $\ker ({\hat
W{}^{(k)}}^T\nabla g(x^{(k)}))$.
\par \hangindent\parindent\hangafter 1\noindent
5. Compute the Newton direction $s_x^{(k)}$ by solving the linear system
\begin{align*}
  &{\hat Z{}^{(k)}}^T\big(\nabla
  \Lag_{\hat\lambda^{(k)}}(x^{(k)})s_x+\Lag_{\hat\lambda^{(k)}}(x^{(k)})\big)=0\\
  &{\hat W{}^{(k)}}^T\big(g(x^{(k)})+\nabla g(x^{(k)})s_x-\hat d^{(k)}\big)=0.
\end{align*}
and set $x^{(k+1)}:=x^{(k)}+s_x^{(k)}$.
\par \hangindent2em\hangafter 1\noindent
6. Increase $k:=k+1$ and go to step 2.
\end{algorithm}

\begin{theorem}\label{ThConvGEAppl}
  Assume that $\xb$ solves \eqref{EqGEAppl}  and both Assumption \ref{AssNonDegen} and Assumption
  \ref{AssSecOrd} are fulfilled. Then there is a neighborhood $U$ of $\xb$ such that for every
  starting point $x^{(0)}\in U$ Algorithm \ref{AlgNewtonGEAppl} either stops after finitely many
  iterations at a solution of \eqref{EqGEAppl} or produces a sequence $x^{(k)}$ converging
  superlinearly to $\xb$.
\end{theorem}
\begin{proof}
  Follows from Theorem \ref{ThConvGenSemiSmmooth1} and Proposition \ref{PropCalG}.
\end{proof}
We now want to compare Algorithm \ref{AlgNewtonGEAppl} with the usual Josephy-Newton method for solving
\eqref{EqGEApplJos}. Given an iterate $(x^{(k)},\lambda^{(k)})$, the new iterate
$(x^{(k+1)},\lambda^{(k+1)})$ is computed as solution of the partially linearized system
\begin{equation}\label{EqGEApplJosLin}
  \myvec{0\\0}\in\myvec{\Lag_{\lambda^{(k)}}(x^{(k)})+ \nabla
  \Lag_{\lambda^{(k)}}(x^{(k)})(x^{(k+1)}-x^{(k)}) +\nabla
  g(x^{(k)})^T(\lambda^{(k+1)}-\lambda^{(k)})\\
  (g(x^{(k)})+\nabla g(x^{(k)})(x^{(k+1)}-x^{(k)}),\lambda^{(k+1)})}-\{0\}\times\gph N_D,
\end{equation}
i.e.,
\begin{eqnarray}\label{EqGEApplJosLinx}0&=& f(x^{(k)})+ \nabla
\Lag_{\lambda^{(k)}}(x^{(k)})(x^{(k+1)}-x^{(k)})+\nabla g(x^{(k)})^T\lambda^{(k+1)},\\
\nonumber \lambda^{(k+1)}&\in& N_D(g(x^{(k)})+\nabla
g(x^{(k)})(x^{(k+1)}-x^{(k)})).\end{eqnarray}
In order to guarantee that \eqref{EqGEApplJosLin} and
\eqref{EqGEApplJosLinx}, respectively, are solvable for all $(x^{(k)},\lambda^{(k)})$ close to $(\xb,\lb)$
one has to impose an additional condition on $\tilde F$ given by \eqref{EqGEApplJos}, e.g., metric
regularity of $\tilde F$ around $\big((\xb,\lb),(0,0)\big)$.

On the contrary the approximation step as described in Algorithm \ref{AlgApprStep} requires only the
solution of a strictly convex quadratic programming problem and the Newton step is performed by solving a
linear system. Note that Assumptions \ref{AssNonDegen},\ref{AssSecOrd} do not imply that multifunctions $x\tto
f(x)+\nabla g(x)^TN_D(g(x))$ or $(x,d)\tto F(x,d)$ are metrically regular around $(\xb, 0)$ and
$((\xb,g(\xb)),(0,0))$, respectively.
\begin{example}
  Consider the NCP
  \begin{equation}\label{EqEx}0\in -x-x^2+N_{\R_-}(x)\end{equation}
  having the unique solution $\xb=0$.  Since $g(x)=x$, we have $\nabla g(x)=1$ showing that $(0,0)$
  is non-degenerate with modulus $1$. We obtain $\lb=0$ and Assumption \ref{AssSecOrd} follows
  easily from  $\nabla L_{\lb}(\xb)=-1$. Thus Theorem \ref{ThConvGEAppl} applies and we obtain
  local superlinear convergence of Algorithm \ref{AlgNewtonGEAppl}. Indeed, given $x^{(k)}$, the
  quadratic program $QP(x^{(k)})$ amounts to
  \[\min_{u\in\R} -(x^{(k)}+{x^{(k)}}^2)u+\frac 12 u^2\ \mbox{subject to}\ x^{(k)}+u\leq 0\]
  which has the solution $u=\min\{x^{(k)}+{x^{(k)}}^2, -x^{(k)}\}$ resulting in $\hat
  d^{(k)}=x^{(k)}+\min\{x^{(k)}+{x^{(k)}}^2, -x^{(k)}\}=\min\{2x^{(k)}+{x^{(k)}}^2,0\}$. If
  $2x^{(k)}+{x^{(k)}}^2<0$, then $T_{\R_-}(\hat d^{(k)})=\R$, $\hat\lambda^{(k)}=0$, $\hat
  l^{(k)}=1$, $\hat Z^{(k)}=1$ and the Newton direction $s_x$ is given by
  \[\hat
  Z^{(k)}{}^T\big(\nabla\Lag_{\hat\lambda^{(k)}}(x^{(k)})s_x+\Lag_{\hat\lambda^{(k)}}(x^{(k)})\big)=-(1+2x^{(k)})s_x-(x^{(k)}+{x^{(k)}}^2)=0\
  \Rightarrow\ s_x=-\frac{x^{(k)}+{x^{(k)}}^2}{1+2x^{(k)}}.\]
  This yields $x^{(k+1)}={x^{(k)}}^2/(1+2x^{(k)})$. On the other hand, if
  $2x^{(k)}+{x^{(k)}}^2\geq0$, then $T_{\R_-}(\hat d^{(k)})=\R_-$,
  $\hat\lambda^{(k)}=2x^{(k)}+{x^{(k)}}^2$, $\hat l^{(k)}=0$, $\hat W^{(k)}=1$ and the Newton
  direction $s_x$ is given by
  \[\hat W^{(k)}{}^T(x^{(k)}+s_x)=0\ \Rightarrow\ s_x=-x^{(k)}\]
  resulting in $x^{(k+1)}=0$. Hence we obtain in fact locally quadratic convergence of the sequence
  produced by Algorithm \ref{AlgNewtonGEAppl}.

  Now we want to demonstrate that the Newton-Josephy method does not work for this simple example.
  At the $k$-th iterate the problem \eqref{EqGEApplJosLinx} reads as
  \begin{align*}&0\in -x^{(k)}-{x^{(k)}}^2+ (-1-2{x^{(k)}})(x^{(k+1)}-x^{(k)})+\lambda^{(k+1)}
  =-(1+2{x^{(k)}})x^{(k+1)}+{x^{(k)}}^2+\lambda^{(k+1)},\\
  & 0\leq \lambda^{(k+1)}\perp x^{(k+1)}\leq 0
  \end{align*}
  and this auxiliary problem is not solvable for any $x^{(k)}$ with $0<\vert x^{(k)}\vert\leq \frac
  12$ . The reason is that the mapping $\tilde F$ is not metrically regular at $(\xb,\lb)$.\hfill$\triangle$
\end{example}

\section{Conclusion}
The crucial notion used in developing the new Newton-type method is the \ssstar property which pertains
not only to single-valued mappings (like the standard semi-smoothness) but also to sets and
multifunctions. The second substantial ingredient in this development consists of a novel linearization
of the set-valued part of the considered GE which is performed on the basis of the respective limiting
coderivative. Finally, very important is also the modification of the semismoothness${}^*$ in Definition
\ref{DefSemiSmoothGen} which enables us to proceed even if the considered multifunction is not \ssstar in
the original sense of Definition \ref{DefSemiSmooth}.

The new method contains, apart from the Newton step, also the so-called approximation step, having two
principal goals. Firstly, it ensures that in the next linearization we dispose with a feasible point and,
secondly, it enables us to avoid points (if they exists), where the imposed regularity assumption is
violated. In this way one obtains the local superlinear convergence without imposing restrictive
regularity assumption at the solution point (like the strong BD-regularity in
\cite{QiSun93}).

The application in Section 5 illuminates the fact that the implementation to a concrete class of GEs may
be quite demanding. On the other hand, the application area of the new method seems to be very large. It
includes, among other things, various complicated GEs corresponding to variational inequalities of the
second kind, hemivariational inequalities, etc. Their solution via an appropriate variant of the new
method will be subject of a further research.
\medskip

{\bf Acknowledgements.} The research of the first author was  supported by the Austrian Science Fund
(FWF) under grant P29190-N32. The  research of the second author was supported by the Grant Agency of the
Czech Republic, Project 17-04301S, and the Australian Research Council, Project
DP160100854.

\end{document}